\documentclass[a4paper,11pt,english]{article}
\usepackage{babel}
\usepackage[latin1]{inputenc}
\author{Andrea Tellini \\ \small{Departamento de Matem\'aticas} \\ \small{Universidad Aut\'onoma de Madrid} \\ \small{Campus de Cantoblanco, 28049 Madrid, Spain} \\ \texttt{\small{andrea.tellini@uam.es}}}
\title{\textbf{High multiplicity of positive solutions for superlinear indefinite problems with homogeneous Neumann boundary conditions}}
\date{\today}
\usepackage{amsmath,amsfonts,amssymb}
\usepackage{yhmath,mathrsfs,yfonts,arcs}
\usepackage{amsthm}
\usepackage{graphicx}
\usepackage{enumerate}

\usepackage{lineno}

\usepackage{hyperref}
\hypersetup{colorlinks=true,       
    linkcolor=red,          
    citecolor=green,        
    filecolor=magenta,      
    urlcolor=cyan           
    }

\newtheorem{theorem}{Theorem}[section]
\newtheorem{remark}[theorem]{Remark}
\newtheorem{proposition}[theorem]{Proposition}

\newtheorem{corollary}[theorem]{Corollary}

 \DeclareMathOperator{\sign}{sgn}

\newcommand{\field}[1] {\mathbb{#1}}
\newcommand{\N}{\field{N}}

\newcommand{\R}{\field{R}}

\def\a{\alpha}

\def\e{\varepsilon}

\def\g{\gamma}
\def\G{\Gamma}
\def\l{\lambda}
\def\m{\mu}

\def\O{\Omega}
\def\p{\partial}

\def\S{\Sigma}

\def\t{\theta}
\def\T{\Theta}

\def\ov{\overline}
\def\un{\underline}
\def\ua{\uparrow}
\def\da{\downarrow}

\newcommand{\mc}{\mathcal}

\usepackage[textwidth=16.0cm,textheight=24cm]{geometry}

\begin{document}
\maketitle

\begin{center}
\emph{Truth for Giulio Regeni (1988--2016), a young researcher of my homeland}
\end{center}


\begin{abstract}
	We prove that a class of superlinear indefinite problems with homogeneous Neumann boundary conditions admits an arbitrarily high number of positive solutions, provided that the parameters of the problem are adequately chosen. The sign-changing weight in front of the nonlinearity is taken to be piecewise constant, which allows us to perform a sharp phase-plane analysis, firstly to study the sets of points reached at the end of the regions where the weight is negative, and then to connect such sets through the flow in the positive part. Moreover, we study how the number of solutions depends on the amplitude of the region in which the weight is positive, using the latter as the main bifurcation parameter and constructing the corresponding global bifurcation diagrams.
\end{abstract}



\setcounter{equation}{0}
\section{Introduction}
\label{section1}
We investigate existence and multiplicity of positive solutions for the following Neumann boundary value problem
\begin{equation}
\label{eq11}
  \left\{ \begin{array}{l}
  -u''=\l u+a(t)u^p \quad \hbox{for}\;\; t\in(0,1),\cr
  u'(0)=0=u'(1),\end{array}\right.
\end{equation}
where $p>1$ and $\l<0$ are constants and the weight $a(t)$ is a piecewise constant function of type
\begin{equation}
	\label{eq12}
  a(t):=\left\{ \begin{array}{ll} -c & \hbox{if}\;\;
  t\in(0,\alpha)\cup(1-\alpha,1), \cr
  b & \hbox{if}\;\; t\in[\alpha,1-\alpha]
  \end{array}\right.
\end{equation}
with $\a\in(0,1/2)$, $b> 0$ and $c>0$. We will look for strong solutions of \eqref{eq11} in the sense that they are of class $\mc C^1([0,1])\cap \mc C^2([0,\a)) \cap \mc C^2((\a,1-\a)) \cap \mc C^2((1-\a,1])$, which is the highest possible regularity, and, without further mention, we will consider only positive solutions throughout the whole paper.

Problems like \eqref{eq11}--\eqref{eq12}, where the nonlinearity is of superlinear type, since $p>1$, and the weight function changes sign, are known in the literature as \emph{superlinear indefinite}.
Such kind of problems, with Neumann (and even more general) boundary conditions have been widely treated in the last decades starting from \cite{BCN1,BCN2,ALG} (see also the references therein and \cite{AT,GRLG,LG} for the case of Dirichlet boundary conditions). In these works, by using a variety of mathematical techniques that go from variational methods up to bifurcation and continuation theory, necessary and sufficient conditions for existence and some (low) multiplicity results have been obtained.

Recently, in \cite{LTZ}, a new insight has been given to positive solutions of superlinear indefinite problems like the ones we study here. There, with the difference that Dirichlet inhomogeneous boundary conditions $u(0)=u(1)=M\in(0,+\infty]$ are considered (when $M=+\infty$ the boundary condition has to be understood in the limiting sense and the solutions are referred to as \emph{large} or \emph{blow-up solutions}), it has been proved that the structure of positive solutions can be extremely rich. Indeed, by using a topological shooting technique which, to the best of our knowledge, goes back to \cite{MPZ}, it has been shown that, when $\l$ is sufficiently negative, there exists a specific value of $b$, for which the problem possesses an arbitrarily high number of positive solutions. Moreover, by using $b$ as the main bifurcation parameter (an idea which goes back to \cite{LG}), the structure of the global bifurcation diagrams has been determined.

Another situation was also known to produce high multiplicity of positive solutions for superlinear indefinite problems: precisely, when the weight function $a(t)$ in \eqref{eq11} has $n\in\N^*$ components where it is positive, separated by regions where it is negative. Indeed, according to some numerical observations, it was conjectured in \cite{GRLG} that, for homogeneous Dirichlet boundary conditions and $\l$ sufficiently negative, such a problem admits $2^n-1$ positive solutions. This kind of multiplicity was then proved when $\l=0$ and the negative part of the weight is sufficiently large in \cite{GHZ}, by using a shooting technique (we also mention that the same results have been obtained for large solutions in \cite{BDP}), and, later, in \cite{BGH}, in the PDE case  by means of variational methods. Recently, in \cite{FZ15}, the use of topological degree has allowed the authors to obtain the same kind of results for more general nonlinearities and $\l\sim 0$.

In the case of Neumann boundary conditions, analogous results have been obtained in \cite{B} with shooting techniques and in \cite{FZ17} with the coincidence degree. This similarity in the behavior of superlinear indefinite problem with Dirichlet and Neumann boundary conditions, arises the natural question of whether high multiplicity results in the spirit of \cite{LTZ} can be obtained with the simple weight function of \eqref{eq12}, which has a unique positive component (observe that in such a case, both for Dirichlet and Neumann boundary conditions the results of \cite{GHZ,B,FZ15,FZ17} guarantee the existence of just $2^1-1=1$ positive solution).

In this work we will positively answer this question by using the same topological shooting technique of \cite{LTZ}, which firstly consists in studying separately the sets of points reached in the phase plane by all the solutions of the problem in $(0,\a)$ and $(1-\a,1)$, i.e. where the weight is negative, and then in connecting such sets through the flow in $(\a,1-\a)$, where the weight is positive. For this last point, we perform a careful analysis of the time maps that allow us to establish all the types of connections.

Here, however, contrarily to \cite{LTZ}, we consider homogeneous boundary conditions, which requires a sharper analysis of the solutions of the sublinear parts near $u=0$ (see Theorem \ref{th22}\eqref{th22iv}). In addition, this makes not clear how to find particular values of $b$ to get high multiplicity, as it was the case in \cite{LTZ}, and then let $b$ vary to obtain the structure of the bifurcation diagrams.

To overcome this problem, we use a new approach that consists in using $\a$ as the main bifurcation parameter, regulating in this way the amplitude of the region in which the weight is positive. Firstly, we obtain arbitrarily high multiplicity, when $\l$ is sufficiently negative, for the purely superlinear problem corresponding to $\a=0$ (we point out that similar results have been recently obtained with different techniques in \cite{BGT} for radial solutions in a ball); then, a singular perturbation allows us to obtain high multiplicity for $\a\sim 0$ (see Theorem \ref{th45}). 

Observe that, by integrating the differential equation in \eqref{eq11} and using the boundary conditions, we obtain a necessary condition for the existence of positive solutions, i.e. that $a(t)$ has to be positive on a subset of $(0,1)$ with positive measure. Thus, no solution can exist for $\a=1/2$. In Theorem \ref{th51}, we will show  how the several solutions that we obtain for $\a\sim 0$ are progressively lost as $\a$ increases up to reach the value $1/2$.

Moreover, we determine the structure of the bifurcation diagrams in $\a$. A remarkable novel result is that, when $\l$ is sufficiently negative, the bifurcation diagrams always exhibit several isolated bounded components (see Theorem \ref{th51} and Figure \ref{fig5}), whose number can be arbitrarily high. On the contrary, when $b$ is used as the main bifurcation parameter, this phenomenon is typically related to the presence of asymmetric weights, as shown in \cite{LT} for Dirichlet boundary conditions, and does not happen in the case of symmetric weights treated in \cite{LTZ}. Still, considering asymmetric weights in the case of Neumann boundary condition further increases the number of components, as a consequence of the breaking of secondary bifurcation points (see Remark \ref{re52}).

%
%

This work is distributed like follows: in Section \ref{section2}, we study the sublinear problems, i.e. Problem \eqref{eq11}--\eqref{eq12} in $(0,\a)$ and $(1-\a,1)$, where the weight is negative, while, in Section \ref{section3}, we consider the superlinear flow in $[\a,1-\a]$ and introduce the time maps that will be the key point for the construction of solutions. In Section \ref{section4}, we establish exact multiplicity results for the purely superlinear case $\a=0$ and then, through a singular perturbation, we obtain multiplicity results for $\a\sim 0$. Finally, in Section \ref{section5}, we prove some general multiplicity results and provide the structure of the global bifurcation diagrams in $\a$, briefly considering also the case of asymmetric weights.

\setcounter{equation}{0}
\section{The sublinear problem}
\label{section2}

This section is devoted to the study of Problem \eqref{eq11} in $(0,\a)$ and  $(1-\a,1)$, where the weight is negative and the differential equation of \eqref{eq11} reduces to
\begin{equation}
\label{eq20}
-u''(t)=\l u(t)-cu(t)^p.
\end{equation}
We start with the case of $(0,\a)$ and apply a shooting method, i.e. we consider the problem
\begin{equation}
\label{eq21}
  \left\{ \begin{array}{l}
  -u''=\l u-c u^p \quad \hbox{for}\;\; t\in(0,\a),\cr
  u'(0)=0, \quad u(0)=s>0\end{array}\right.
\end{equation}
and study all its positive solutions and their properties as $s$ varies. For every $s>0$, the solution of \eqref{eq21} blows up in finite time, since \eqref{eq20} satisfies the well-known Keller--Osserman condition (see \cite{K,O}), thus we need to guarantee that it is defined in the whole interval $[0,\a]$. This is done in the following proposition, where in addition we prove some important qualitative properties of \eqref{eq21}, like the monotonic dependence on $s$ of the solutions and their derivatives.

\begin{proposition}
\label{pr21}
There exists a unique $s_{\infty}$ such that the unique solution of Problem \eqref{eq21}, denoted by $u_s(t)$,  is defined in $[0,\a]$ if and only if $s\in(0,s_{\infty})$. Moreover:
\renewcommand{\theenumi}{\roman{enumi}}
\begin{enumerate}[(i)]
\item \label{pr21i} for $0<s<s_{\infty}$, $u_s(t)>0$ and $u_s'(t)>0$ for $t\in(0,\a]$;
\item \label{pr21ii} $\lim\limits_{s\da 0}u_s(\a)=0, \quad \lim\limits_{s\ua s_{\infty}}u_s(\a)=+\infty$;
\item \label{pr21iii} if $0<s_1<s_2<s_{\infty}$, then $u_{s_1}(t)<u_{s_2}(t)$ and $u_{s_1}'(t)<u_{s_2}'(t)$ for $t\in(0,\a]$.
\end{enumerate}
\end{proposition}
\begin{proof}
(i) By introducing $v(t):=u'(t)$ and integrating the first order system for $(u(t),v(t))$ associated to \eqref{eq21}, it is easy to see that its solution, which is unique thanks to the Cauchy--Lipschitz theorem, lies on the integral curves defined by
\begin{equation}
\label{eq22}
v^2+\l u^2-\frac{2c}{p+1}u^{p+1}=\l s^2-\frac{2c}{p+1}s^{p+1},
\end{equation}
which have been represented in the phase plane $(u,v)$ in Figure \ref{fig21}, together with the associated flow.

\begin{figure}[ht]
\begin{center}
\includegraphics[scale=0.50]{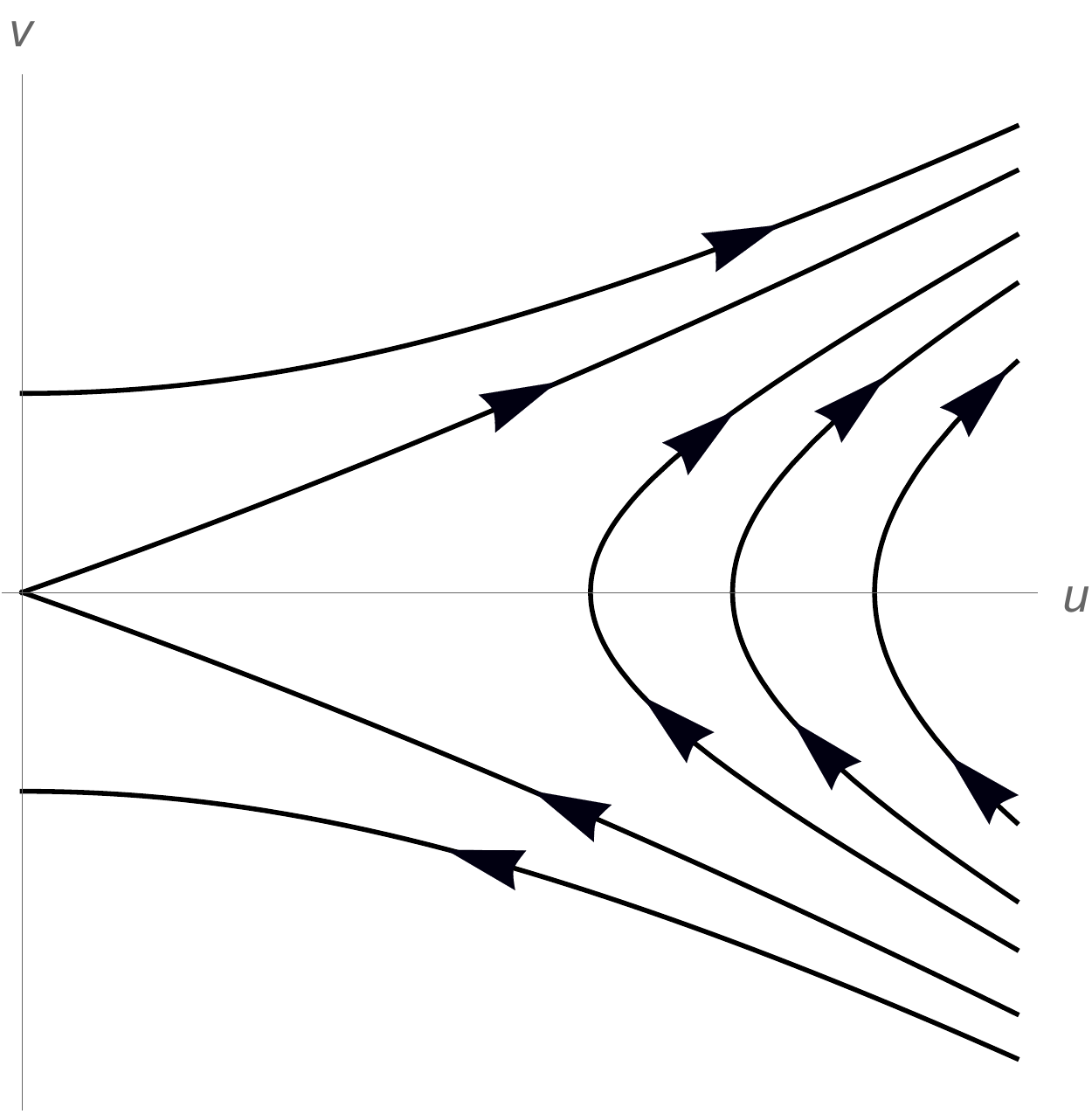} \\
\caption{Phase portrait associated to equation \eqref{eq20}.} \label{fig21}
\end{center}
\end{figure}

By analyzing the flow, it immediately follows that the solution of \eqref{eq21}, if defined, is increasing, thus positive, since it starts on the half-line $\{u>0,\, v=0\}$ (this, \emph{a posteriori}, proves (i)) and, due to \eqref{eq22}, satisfies
\begin{equation}
\label{eq23}
v(t)=\sqrt{-\l(u(t)^2-s^2)+\frac{2c}{p+1}(u(t)^{p+1}-s^{p+1})}.
\end{equation}
To show the existence of $s_{\infty}$, we introduce
\begin{align}
T_{\infty}(s)&:=\int_s^{+\infty}\frac{du}{\sqrt{-\l(u^2-s^2)+\frac{2c}{p+1}(u^{p+1}-s^{p+1})}} \notag\\
&=\int_1^{+\infty}\frac{d\xi}{\sqrt{-\l(\xi^2-1)+\frac{2c}{p+1}s^{p-1}(\xi^{p+1}-1)}}<+\infty, \label{eq24}
\end{align}
which, in view of \eqref{eq23}, measures the time needed for the solution of \eqref{eq21} to blow up at $t=T_{\infty}(s)$, i.e. to satisfy \eqref{eq20} in $t\in(0,T_{\infty}(s))$ and
\begin{equation}
\label{eq25}
\lim_{t\ua T_{\infty}(s)}u(t)=+\infty.
\end{equation}
From \eqref{eq24} it follows that $T_{\infty}(s)$ is continuous, decreasing in $s$ and satisfies
\begin{equation*}
\lim_{s\da0}T_{\infty}(s)=+\infty, \qquad \lim_{s\ua\infty}T_{\infty}(s)=0,
\end{equation*}
thus there exists a unique value of $s$, denoted by $s_{\infty}$, such that $T_{\infty}(s_{\infty})=\a$ and which possesses the desired properties.

(ii) Relations \eqref{pr21ii} follow from the continuous dependence theorem on initial conditions: the first relation by observing that, for $s=0$, Problem \eqref{eq21} admits the unique solution $u=0$, while the second relation by using the definition of $s_{\infty}$ and \eqref{eq25}. 

(iii) Finally, conditions \eqref{pr21iii} can be obtained from some comparison principles, as shown hereafter. Take $0<s_1<s_2<s_{\infty}$ and assume by contradiction that $u_{s_1}(\tilde{t})\geq u_{s_2}(\tilde{t})$ for some $\tilde{t}\in(0,\a]$. Then, by continuity, there exists $\t\in(0,\tilde{t}]$ such that $u_{s_1}(\t)=u_{s_2}(\t)=:U>0$ and, by the uniqueness of positive solutions of the boundary value problem
\begin{equation*}
  \left\{ \begin{array}{l}
  -u''=\l u-c u^p \quad \hbox{for}\;\; t\in(0,\t),\cr
  u'(0)=0, \quad u(\t)=U\end{array}\right.
\end{equation*}
(see, for example, \cite{CC}), we obtain $u_{s_1}=u_{s_2}$, which is a contradiction. As for the remaining one, we observe that
\begin{equation*}
u_{s_1}''(0)=-\l s_1+c s_1^p<-\l s_2+c s_2^p=u_{s_2}''(0),
\end{equation*}
and, as a consequence, $u_{s_1}'<u_{s_2}'$ in a right neighborhood of $t=0$. Now, if we assume that $u_{s_1}'(\tilde t)\geq u_{s_2}'(\tilde t)$ for some $\tilde{t}\in(0,\a]$, by continuity there exists $\t\in(0,\tilde t]$ such that $u_{s_1}'(\t)= u_{s_2}'(\t)=: V>0$ and, by the uniqueness of positive solutions of the boundary value problem
\begin{equation*}
  \left\{ \begin{array}{l}
  -u''=\l u-c u^p \quad \hbox{for}\;\; t\in(0,\t),\cr
  u'(0)=0, \quad u'(\t)=V\end{array}\right.
\end{equation*}
(see again \cite{CC}), we obtain once more the contradiction $u_{s_1}=u_{s_2}$.
%
\end{proof}
If we denote by $\S_0$ the set  of all the nonnegative solutions of \eqref{eq21}, then, according to the previous proposition, we have
\begin{equation*}
\S_0:=\{u(t): u  \text { solves \eqref{eq21} with } s\in[0,s_{\infty})\}.
\end{equation*}
Moreover, we introduce the following set in the phase plane $\R^2$
\begin{equation*}
\G_0:=\{(u(\a),u'(\a)): u\in\S_0\},
\end{equation*}
which will play a fundamental role in the construction of the solutions of \eqref{eq11}. The next result provides the key properties of this set that will be required hereafter.

\begin{theorem}
\label{th22}
There exists a function $y$ of class $\mc C^1([0,+\infty))\cap\mc C^{\infty}((0,+\infty))$ such that
\begin{enumerate}[(i)]
\item \label{th22i} $\G_0=\{\left(x,y(x)\right), \text{ for } x\geq 0\}$;
\item \label{th22ii} $y(0)=0$ and $y(x)>0$ for $x>0$;
\item \label{th22iii} $y'(x)>0$ for $x>0$;
\item \label{th22iv} $y'(0)=\sqrt{-\l}\tanh\left(\sqrt{-\l}\a\right)$;
\item \label{th22v} $\lim\limits_{x\to+\infty}\displaystyle\frac{y(x)}{x}=+\infty$.
\end{enumerate}
\end{theorem}
\begin{proof}
(i)-(ii)-(iii) Consider the map
\begin{align*}
\mc P:[0,s_\infty)&\to\R^2, \\
s & \mapsto \left(u_s(\a),u_s'(\a)\right),
\end{align*}
where $u_s$ is the solution of \eqref{eq21}. $\mc P$ is continuous by the continuous dependence theorem of solutions with respect to initial data. Proposition \ref{pr21}\eqref{pr21ii} implies that $\pi_u\mc P([0,s_\infty))=\pi_u \G_0=[0,+\infty)$, where $\pi_u$ denotes the projection on the $u$ component of the phase plane $\R^2$.

Property \eqref{th22ii} is a direct consequence of Proposition \ref{pr21}\eqref{pr21i}. To show \eqref{th22i} and \eqref{th22iii}, we observe that $\mc P$ is $\mc{C}^{\infty}$ for $s\in(0,s_{\infty})$ by successive applications of the differentiable dependence theorem of solutions with respect to initial data. Moreover, Proposition \ref{pr21}\eqref{pr21iii} guarantees that, for $x>0$ (so that $s>0$), $\G_0$ can be locally parameterized, with respect to the $u$ variable in the phase plane, as the graph of an increasing function of class $\mc{C}^{\infty}$, and we obtain $y$ by gluing all these local graphs, with standard arguments from differential geometry.

(iv) To show that $y$ is differentiable also for $x=0$, we consider the parametrization of $\G_0$ given by $\mc P$ and, by using the chain rule and recalling that $y(0)=0$, we have
\begin{equation*}
	y'(0)=\frac{d}{ds}u_s'(\a)\big\rvert_{s=0}\left(\frac{d}{ds}u_s(\a)\big\rvert_{s=0}\right)^{-1}=\lim_{s\da 0}\frac{u_s'(\a)}{s}\left(\lim_{s\da 0}\frac{u_s(\a)}{s}\right)^{-1}.
\end{equation*}
Let us denote by $L'$ and $L$ the first and the second limit in the right-hand side of the previous relation, respectively. We start by computing $L$: from \eqref{eq23} we have
\begin{align*}
\a&=\int_0^\a\frac{u'(t)\, dt}{\sqrt{-\l(u(t)^2-s^2)+\frac{2c}{p+1}(u(t)^{p+1}-s^{p+1})}} \\
&=\int_s^{u_s(\a)}\frac{du}{\sqrt{-\l(u^2-s^2)+\frac{2c}{p+1}(u^{p+1}-s^{p+1})}} \\
&=\int_1^{\frac{u_s(\a)}{s}}\frac{d\xi}{\sqrt{-\l(\xi^2-1)+\frac{2c}{p+1}s^{p-1}(\xi^{p+1}-1)}},
\end{align*}
and, by taking $s\da 0$ in the last relation, we obtain
\begin{equation*}
\sqrt{-\l}\a=\int_1^L\frac{d\xi}{\sqrt{\xi^2-1}}=\log\left(\sqrt{L^2-1}+L\right),
\end{equation*}
which gives $L=\cosh(\sqrt{-\l}\a)$.

On the other hand, by differentiating \eqref{eq20} with respect to $t$, we obtain that $v_s(t):=u_s'(t)$ satisfies the following Cauchy problem
\begin{equation*}
  \left\{ \begin{array}{l}
  -v''=\left(\l-cp u_s(t)^{p-1}\right) v \quad \hbox{for}\;\; t\in(0,\a),\cr
  v(0)=0, \quad v'(0)=-\l s+c s^p.\end{array}\right.
\end{equation*}
By Proposition \ref{pr21}\eqref{pr21i} we have that $v_s(t)>0$ and $0<u_s(t)<u_s(\a)$ for all $t\in(0,\a)$. Therefore, $v_s$ satisfies the following differential inequalities
\begin{equation*}
-\l v(t)\leq v''(t)\leq\left(-\l+cp u_s(\a)^{p-1}\right) v(t),
\end{equation*}
which, by Kamke's comparison theorem (see \cite{Ka}), entail that
\begin{equation}
\label{eq26}
\un v(\a)\leq v_s(\a)\leq\ov v(\a),
\end{equation}
where $\un v $ and $\ov v$ are, respectively, the solutions of the linear Cauchy problems, considered for $t\in(0,\a)$,
\begin{equation*}
  \left\{ \begin{array}{l}
  -v''=\l v, \cr
  v(0)=0, \quad v'(0)=-\l s+c s^p\end{array}\right. 
  \quad \text{ and } \quad 
    \left\{ \begin{array}{l}
  -v''=(\l-cp u_s(\a)^{p-1}) v, \cr
  v(0)=0, \quad v'(0)=-\l s+c s^p,\end{array}\right.
\end{equation*}
i.e.
\begin{align*}
\un v(t)&=\frac{-\l s+cs^p}{\sqrt{-\l}}\sinh\left(\sqrt{-\l} t\right), \\
 \ov v(t)&=\frac{-\l s+cs^p}{\sqrt{-\l+cpu_s(\a)^{p-1}}}\sinh\left(\sqrt{-\l+cpu_s(\a)^{p-1}} t\right).
\end{align*}
Recalling Proposition \ref{pr21}\eqref{pr21ii} and \eqref{eq26}, we obtain $L'=\sqrt{-\l}\sinh\left(\sqrt{-\l}\a\right)$, thus $y'(0)=L'/L=\sqrt{-\l}\tanh\left(\sqrt{-\l}\a\right)$, which concludes the proof of \eqref{th22iv}.

(v) We evaluate \eqref{eq22} at $t=\a$, obtaining
\begin{equation*}
y(x(s))^2+\l x(s)^2-\frac{2c}{p+1}x(s)^{p+1}=\l s^2-\frac{2c}{p+1}s^{p+1},
\end{equation*}
where we have set $x(s):=u_s(\a)$. By dividing this relation by $x(s)^2$, taking the limit as $s\ua s_\infty$ and recalling Proposition \ref{pr21}\eqref{pr21ii}, we have that the right-hand side goes to $0$ and \eqref{th22v} follows.
\end{proof}

We now consider the Cauchy sublinear problem in $(1-\a,1)$, which reads
\begin{equation}
\label{eqsub2}
  \left\{ \begin{array}{l}
  -u''=\l u-c u^p \quad \hbox{for}\;\; t\in(1-\a,1),\cr
  u'(1)=0, \quad u(1)=s>0,\end{array}\right.
\end{equation}
and we have this immediate counterpart of Proposition \ref{pr21}.

\begin{corollary}
\label{co23}
The unique solution of Problem \eqref{eqsub2}, denoted by $u_s(t)$, is defined in $[1-\a,1]$ if and only if $s\in(0,s_{\infty})$, where $s_{\infty}$ is the one of Proposition \ref{pr21}. Moreover:
\renewcommand{\theenumi}{\roman{enumi}}
\begin{enumerate}[(i)]
\item \label{co23i} for $0<s<s_{\infty}$, $u_s(t)>0$ and $u_s'(t)<0$ in $[1-\a,1)$;
\item \label{co23ii} $\lim\limits_{s\da 0}u_s(\a)=0, \qquad \lim\limits_{s\ua s_{\infty}}u_s(\a)=+\infty$;
\item \label{co23iii} if $0\!<\!s_1\!<\!s_2\!<\!s_{\infty}$, then $u_{s_1}(t)<u_{s_2}(t)$ and $u_{s_1}'(t)>u_{s_2}'(t)$ for $t\in[1-\a,1)$.
\end{enumerate}
\end{corollary}
\begin{proof}
Performing the change of variables $\tilde u(t)=u(1-t)$ for $t\in[1-\a,1]$, we have that there is a one-to-one correspondence between \eqref{eqsub2} and \eqref{eq21}, with the unique difference of an opposite sign of the first derivatives with respect to time. As a consequence, the result immediately follows from Proposition \ref{pr21}.
\end{proof}

In analogy with the notation used above, we introduce 
\begin{align*}
\S_1&:=\{u(t): u  \text { solves \eqref{eqsub2} with } s\in[0,s_{\infty})\}, \\
\G_1&:=\{(u(1-\a),u'(1-\a)): u\in\S_1\},
\end{align*}
and, by performing the same change of variables as in the proof of Corollary \ref{co23}, we have

\begin{corollary}
\label{co24}
If $y$ is the function constructed in Theorem \ref{th22}, then
\begin{equation*}
\G_1=\{(x,-y(x)), \text{ for } x\geq 0\}.
\end{equation*}
\end{corollary}

Since we will use $\a$ as a parameter in the following, we conclude this section by establishing some properties of the function $y(x,\a)$, the one given by Theorem \ref{th22}, where we have explicitly pointed out its dependence also with respect to this parameter.

\begin{proposition}
\label{pr25}
(i) The function $y(x,\a)$ converges to $0$ as $\a\da 0$ uniformly for $x$ in compact sets of $\R^+$. Moreover, for $x>0$, it is differentiable with respect to $\a\in[0,+\infty)$ and satisfies
\begin{equation}
\label{dydalpha}
\frac{\p y}{\p \a}\bigg\rvert_{\a=0}=-\l x+cx^p>0.
\end{equation}
(ii) The function $\frac{\p y(x,\a)}{\p x}$ converges to $0$ as $\a\da 0$ uniformly for $x$ in compact sets of $\R^+$. Moreover, for $x>0$, it is differentiable with respect to $\a\in[0,+\infty)$ and satisfies
\begin{equation}
	\label{dy'dalpha}
		\frac{\p}{\p \a}\left(\frac{\p y}{\p x}\right)\bigg\rvert_{\a=0}=-\l +cpx^{p-1}>0.
\end{equation}
(iii) The function $\frac{\p^2 y(x,\a)}{\p x^2}$ converges to $0$ as $\a\da 0$ uniformly for $x$ in compact sets of $\R^+$. Moreover, for $x>0$, it is differentiable with respect to $\a\in[0,+\infty)$ and satisfies
\begin{equation*}
\label{dy''dalpha}
\frac{\p}{\p \a}\left(\frac{\p^2 y}{\p x^2}\right)\bigg\rvert_{\a=0}= cp(p-1)x^{p-2}>0.
\end{equation*}
In particular, for every $x\in\R^+$, $y''(x)>0$  for $\a\sim 0$.
\begin{proof}
(i) Consider $x>0$ and recall that $y(x,\a)$ is the value of $u'(\a)$, with $u(t)$ being the solution of
\begin{equation}
\label{eqxfisso}
  \left\{ \begin{array}{l}
  -u''=\l u-c u^p \quad \hbox{for}\;\; t\in(0,\a),\cr
  u'(0)=0, \quad u(\a)=x.\end{array}\right.
\end{equation}
By performing the change of variable $\tilde u(\tilde t)=u(t)$, where $t=\a\tilde t$, we obtain from \eqref{eqxfisso} that $\tilde u$ satisfies
\begin{equation}
\label{eqxfissoscaled}
  \left\{ \begin{array}{l}
  -\tilde u''=\a^2\left(\l \tilde u-c \tilde u^p\right) \quad \hbox{for}\;\; \tilde t\in(0,1),\cr
  \tilde u'(0)=0, \quad \tilde u(1)=x,\end{array}\right.
\end{equation}
where $'$ now denotes $d/d\tilde t$. Moreover, we have
\begin{equation}
\label{eqderivative}
\frac{du}{dt}(\a)=\frac{1}{\a}\frac{d\tilde u}{d\tilde t}(1).
\end{equation}
Thanks to the differentiable dependence theorem  with respect to parameters applied to Problem \eqref{eqxfissoscaled}, $\tilde u$ and its derivatives with respect to $\tilde t$ depend differentiably on $\a$, thus $\tilde u(\tilde t, \a)$ converges, as $\a\da 0$, uniformly in $\tilde t\in[0,1]$ to the unique solution of \eqref{eqxfissoscaled} for $\a=0$, i.e. the constant function $\tilde u(\tilde t)=x$. This implies the locally uniform convergence of $y(x,\a)$ to $0$  as $\a\da 0$.

From \eqref{eqderivative} and again the differentiable dependence theorem,  the function $\a\mapsto\frac{du}{dt}(\a)$ is differentiable for $\a>0$. For the differentiability in $\a=0$, after setting $y(x,0)=0$ by continuity, observe that
\begin{equation*}
\frac{\p y}{\p \a}\bigg\rvert_{\a=0}=\lim_{\a\da 0}\frac{y(x,\a)-y(x,0)}{\a}=\lim_{\a\da 0}\frac{\tilde u'(1)}{\a^2}=\lim_{\a\da 0}\frac{\int_0^1\tilde u''(\tilde t)\,d\tilde t}{\a^2}=-\l x+cx^p,
\end{equation*}
as desired.

(ii) By differentiating \eqref{eqxfisso} with respect to $x$, we have that $\frac{\p y(x,\a)}{\p x}=\xi'(\a)$,
where $\xi$ is the unique solution of
\begin{equation*}
	\label{eqlinearizzata}
	\left\{ \begin{array}{l}
		-\xi''=\left(\l -cp u^{p-1}\right)\xi \quad \hbox{for}\;\; t\in(0,\a),\cr
		\xi'(0)=0, \quad \xi(\a)=1\end{array}\right.
\end{equation*}
with $u$ being the solution of \eqref{eqxfisso}. After the same change of variables as above, we obtain that $\frac{\p y(x,\a)}{\p x}=\frac{\tilde\xi'(1)}{\a}$, where $\tilde\xi$ solves 
\begin{equation}
	\label{eqlinearizzatascaled}
	\left\{ \begin{array}{l}
		-\tilde\xi''=\a^2\left(\l -cp \tilde u^{p-1}\right)\tilde\xi \quad \hbox{for}\;\; \tilde t\in(0,1),\cr
		\tilde\xi'(0)=0, \quad \tilde\xi(1)=1\end{array}\right.
\end{equation}
with $\tilde u$ being the solution of \eqref{eqxfissoscaled}. Reasoning as above, we have that $\tilde\xi$ converges, uniformly in $[0,1]$, to the constant function $1$ as $\a\da 0$ and
\begin{equation}
\label{eq:2.16}
	\frac{\tilde\xi'(1)}{\a}=\frac{1}{\a}\int_0^1\tilde\xi''(\tilde t)\,d\tilde t=\a\int_0^1\left(-\l+cp\tilde u^{p-1}(\tilde t)\right)\tilde\xi(\tilde t)\,d\tilde t,
	\end{equation}
	which converges to $0$ as $\a\to 0$. The differentiability of $\a\mapsto\p y/\p x$ follows from \eqref{eq:2.16}, recalling that $\frac{\p y(x,\a)}{\p x}=\frac{\tilde\xi'(1)}{\a}$. Finally, \eqref{dy'dalpha} can be obtained by differentiating directly the right-hand side of \eqref{eq:2.16}, observing that $\p_\a\tilde u(\tilde t)$ and $\p_\a\tilde \xi(\tilde t)$, which respectively  satisfy the boundary value problems obtained by differentiating \eqref{eqxfissoscaled} and \eqref{eqlinearizzatascaled} with respect to $\a$, converge  to $0$ uniformly for $\tilde t\in[0,1]$ as $\a\to 0$.
	
	(iii) The proof of (iii) follows the same lines of that of (ii), by observing that $\frac{\p^2 y(x,\a)}{\p x^2}=\frac{\tilde\eta'(1)}{\a}$, where $\tilde{\eta}$ solves
	\begin{equation*}
	\label{eqeta}
	\left\{ \begin{array}{l}
	-\tilde\eta''=\a^2\left[-cp\left(p-1\right) \tilde u^{p-2}\tilde\xi^2+\left(\l -cp \tilde u^{p-1}\right)\tilde\eta\right] \quad \hbox{for}\;\; \tilde t\in(0,1),\cr
	\tilde\eta'(0)=0, \quad \tilde\eta(1)=0,\end{array}\right.
	\end{equation*}
	where $\tilde u$ and $\tilde{\xi}$ are solutions of \eqref{eqxfissoscaled} and \eqref{eqlinearizzatascaled} respectively.
\end{proof}
\end{proposition}

\setcounter{equation}{0}
\section{Geometry of the superlinear problem and general existence result}
\label{section3}
In this section we study equation \eqref{eq11} in $(\a,1-\a)$, where it reduces to
\begin{equation}
\label{eq31}
-u''(t)=\l u(t)+bu^p(t).
\end{equation}
First of all, we observe that it admits the first integral
\begin{equation}
\label{eq32}
E(u,v):=v^2+\l u^2+\frac{2b}{p+1}u^{p+1},
\end{equation}
which is constant along the trajectories of the solutions of \eqref{eq31}. Moreover, since $\l<0$ and $b>0$, apart from $(0,0)$, there is another nonnegative equilibrium, denoted by $(\O,0)$, where
\begin{equation*}
\O=\left(\frac{-\l}{b}\right)^\frac{1}{p-1}.
\end{equation*}
The matrix $D^2 E(0,0)$ has eigenvalues of opposite signs, therefore there is a homoclinic orbit in the phase plane passing through $(0,0)$, which will be denoted by $\g_h$. It follows from \eqref{eq32} that it can be parameterized as $(u,\pm v_h(u))$, for $u\in(0,u_h]$, where
\begin{equation}
\label{eq33}
v_h(u)=\sqrt{-\l u^2- \frac{2b}{p+1}u^{p+1}}
\end{equation}
and
\begin{equation}
\label{eq34}
u_h:=\O\left(\frac{p+1}{2}\right)^{\frac{1}{p-1}}
\end{equation}
(observe that $\g_h$ intersects the $u$-axis in $(u_h,0)$).

On the other hand, $D^2 E(\O,0)$ is positive definite, meaning that the equilibrium $(\O,0)$ is a center which is surrounded by closed orbits, up to the homoclinic $\g_h$.

We are now interested in the superposition of this geometry with the sets $\G_0$ and $\G_1$ introduced in Section \ref{section2}, and our goal is to connect $\G_0$ to $\G_1$ through the flow induced by \eqref{eq31} for $t\in(\a,1-\a)$. Indeed, if we find $(x,y(x))\in\G_0$ such that the unique solution $u_c(t)$ of
\begin{equation*}
  \left\{ \begin{array}{l}
  -u''=\l u+b u^p \quad \hbox{for}\;\; t\in(\a,1-\a),\cr
  u(\a)=x, \quad u'(\a)=y(x)>0\end{array}\right.
\end{equation*}
satisfies $(u_c(1-\a),u_c'(1-\a))\in\G_1$, then, if we denote by $u_l(t)$ and $u_r(t)$ the unique (thanks to the results of Proposition \ref{pr21} and Corollary \ref{co23}) solutions of
\begin{equation*}
  \left\{ \begin{array}{l}
  -u''=\l u-c u^p \quad \hbox{in}\;\; (0,\a),\cr
  u(\a)=x, \cr u'(\a)=y(x) \end{array}\right. 
  \quad \text{ and } \quad
  \left\{ \begin{array}{l}
    -u''=\l u-c u^p \quad \hbox{in}\;\; (1-\a,1),\cr
  u(1-\a)=u_c(1-\a), \cr
  u'(1-\a)=u_c'(1-\a),\end{array}\right.
\end{equation*}
respectively, then
\begin{equation*}
u(t):=\left\{ \begin{array}{ll}
u_l(t) & \text{ for } t\in[0,\a), \\
u_c(t) & \text{ for } t\in[a,1-\a], \\
u_r(t) & \text{ for } t\in(1-\a,1] 
\end{array}\right.
\end{equation*}
is a strong solution of \eqref{eq11}.

Going back to the geometry in the phase plane, Theorem \ref{th22}\eqref{th22iii} implies that $(\O,0)$ lies at a positive distance from $\G_0$, entailing that closed orbits near $(\O,0)$ do not intersect $\G_0$. On the other hand, from Theorem \ref{th22}\eqref{th22iv} and \eqref{eq33}, we have
\begin{equation*}
y'(0)=\sqrt{-\l}\tanh\left(\sqrt{-\l}\a\right)<\sqrt{-\l}=v_h'(0),
\end{equation*}
thus  $\G_0$ lies inside $\g_h$ for $x\sim 0$ and all $\a\in(0,1/2)$, and closed orbits of \eqref{eq31} near $\g_h$ do instead intersect $\G_0$, since they have a vertical tangent on the $u$-axis. By continuity there exists a critical orbit, denoted by $\g_t$, which intersects $\G_0$ being tangent, while orbits between $\g_t$ and $\g_h$ will be secant to $\G_0$ and the ones between $\g_t$ and $(\O,0)$ will not touch $\G_0$. Observe that the same patterns hold true for $\G_1$, as Corollary \ref{co24} says that this curve is obtained by reflecting $\G_0$ with respect to the $u$-axis and the orbits of \eqref{eq31} are symmetric with respect to such an axis, as a consequence of \eqref{eq32}.

In the following, we assume that $\g_t$ is the unique orbit which is tangent to $\G_0$ at some point, that such tangency point, denoted by $(x_t,y(x_t))$ is simple, that all the orbits between $\g_t$ and $\g_h$ intersect $\G_0$ in exactly two points and that all the exterior orbits to $\g_h$ intersect $\G_0$ in exactly one point. In particular, we assume that $\g_h$ intersects $\G_0$ in two points: $(0,0)$ and another one, whose abscissa will be denoted by $x_h>0$.

This situation occurs at least for $\a\sim 0$ as a consequence of the convexity of $\G_0$ given by Proposition \ref{pr25}(iii) and the concavity of the upper part of the closed orbits of \eqref{eq31} surrounding $(\O,0)$. 
Should this not be the case for all $\a\in(0,1/2)$, the existence and multiplicity results presented hereafter would not change, the only difference being that more solutions or components in the bifurcation diagrams might be present.

By taking into account all the features described above, a possible configuration in the phase plane is represented in Figure \ref{fig31}.

\begin{figure}[ht]
\begin{center}
\includegraphics[scale=0.65]{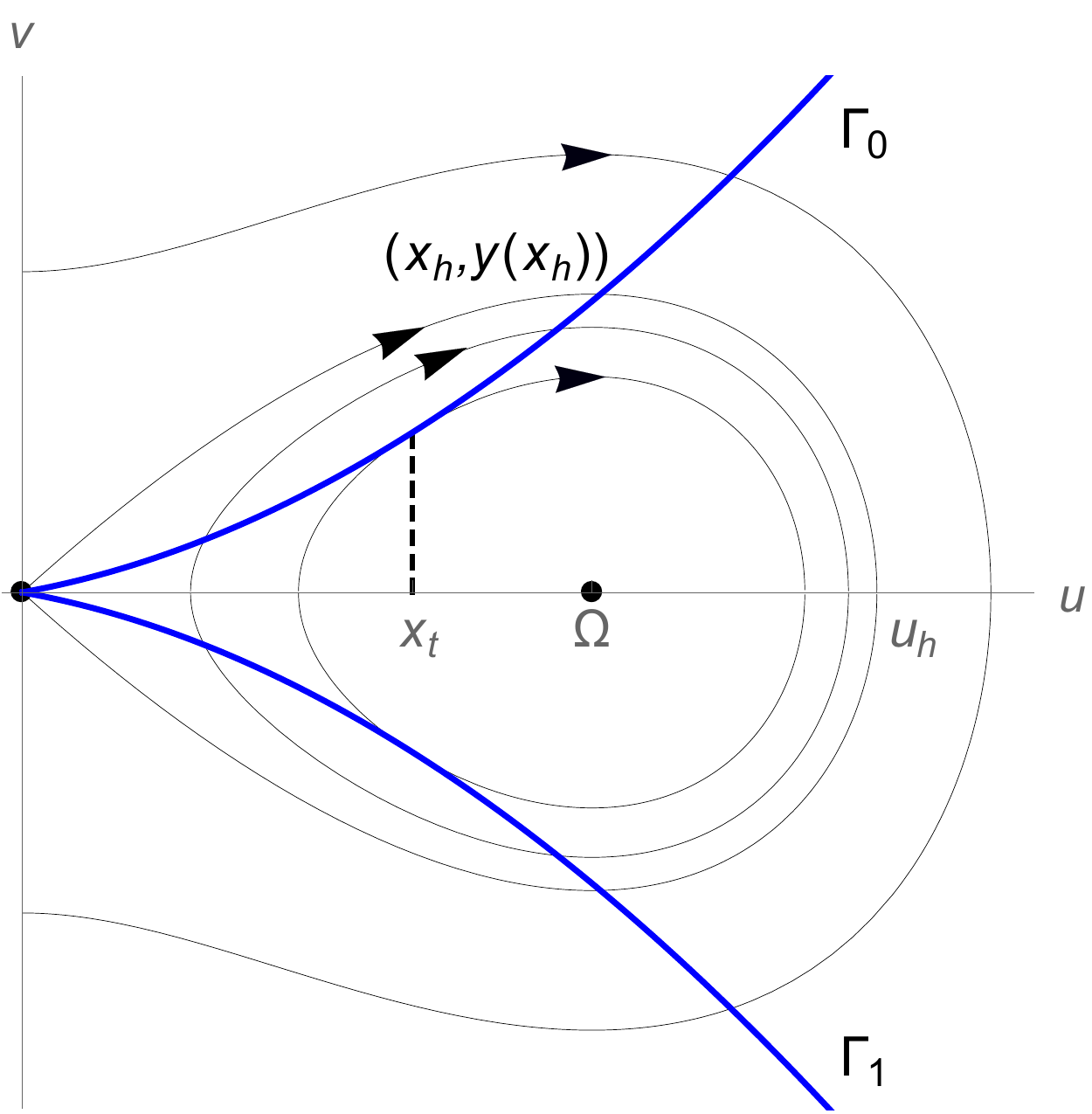} \\
\caption{Geometry of the phase plane in the superlinear part, together with the curves $\G_0$ and $\G_1$.} \label{fig31}
\end{center}
\end{figure}

We conclude this section with  a first general existence result for Problem \eqref{eq11}, which is valid for all the values of the parameters that we consider. Even though it is a particular case of \cite[Theorem 6]{BCN1} and \cite[Theorem 7.1]{ALG},  we present it since we prove it with the specific topological shooting techniques of this work, which will later allow us to obtain  our high multiplicity results.

\begin{theorem}
\label{th31} 
Problem \eqref{eq11} admits (at least) one solution for all the values of the parameters in the considered ranges, in particular for every $\l<0$ and $\a\in(0,1/2)$.
\end{theorem}
\begin{proof}
For every $x>0$ we define 
\begin{align}
\label{eq35}
\tau_s(x)&:=2\int_x^{M(x)}\frac{du}{\sqrt{\l(M(x)^2-u^2)+\frac{2b}{p+1}(M(x)^{p+1}-u^{p+1})}}, \notag \\
&=2\int_\frac{x}{M(x)}^1\frac{d\xi}{\sqrt{\l(1-\xi^2)+\frac{2b}{p+1}M(x)^{p-1}(1-\xi^{p+1})}},
\end{align}
where $M(x)>\max\{x,\O\}$ denotes the maximum abscissa of the orbit satisfying 
\begin{equation}
	\label{eqenergy}
E(u,v)=y(x)^2+\l x^2+\frac{2b}{p+1}x^{p+1}.
\end{equation}
It has to be remarked that, thanks to the symmetry of Problem \eqref{eq11}, $\tau_s(x)$ measures the time needed to connect the point $(x,y(x))\in\G_0$ to $(x,-y(x))\in\G_1$ for the first time. As a consequence, the above discussion shows that every point in $\tau_s^{-1}(1-2\a)$ correspond to a solution of \eqref{eq11}.

By looking at Figure \ref{fig31}, it is apparent that $M(x)\to u_h$ as $x\da 0$, thus, by taking \eqref{eq34} into account, we have
\begin{equation*}
\lim_{x\da 0}\tau_s(x)=\frac{2}{\sqrt{-\l}} \int_0^1\frac{d\xi}{\sqrt{\xi^2-\xi^{p+1}}}=+\infty,
\end{equation*}
while
\begin{equation*}
\lim_{x\ua+\infty}\tau_s(x)=0
\end{equation*}
since $M(x)\to+\infty$ in this case, and the integrand in \eqref{eq35} tends uniformly to $0$. As $\tau_s(x)$ is continuous, we obtain the existence of (at least) one value of $\ov x$ such that $\tau_s(\ov x)=1-2\a$, which gives the desired solution of \eqref{eq11}.
\end{proof}

The above discussion on to the way of constructing solutions to Problem \eqref{eq11} and the geometry of the phase plane motivate the introduction of the following functions, which measure the time needed to connect $\G_0$ to $\G_1$ through the flow of \eqref{eq31} in all the possible ways.

We set $D_1=D_1(\a):=(0,x_t(\a))$, $D_2=D_2(\a):=(x_t(\a),x_h(\a))$ and $D_3=D_3(\a):=[x_h(\a),+\infty)$, and, for every $j\in\N^*$ and $x\in D_1\cup D_2$, we define the maps $\tau_j(x)$ as the time needed to reach $\G_1$ exactly for the $j$th time, starting from $(x,y(x))\in\G_0$ and moving along the orbit given by \eqref{eqenergy}. We also define $\tau_1(x)$ analogously for $x\in D_3$, while, for $x\in D_1\cup D_2\cup\{x_t\}$ we denote by $\tau(x)$ the period of the orbit through $(x,y(x))$. Such functions are continuous by the continuous dependence theorem and, moreover, satisfy the following properties. 

\begin{proposition}
	\label{pr32}
	For $x$ lying in the proper domain of definition and every $j\in\N^*$, the following properties hold true:
	\renewcommand{\theenumi}{\roman{enumi}}
	\begin{enumerate}[(i)]
	
	\item $\tau_{2j-1}(x)=\tau_1(x)+(j-1)\tau(x)$ and $\tau_{2j}(x)=\tau_2(x)+(j-1)\tau(x)$;
	
	\item $\tau_{j+1}(x)>\tau_j(x)$;
	
	\item $\lim\limits_{x\da 0}\tau_j(x)=+\infty=\lim\limits_{x\ua x_h}\tau_{j+1}(x)$;

	\item $\lim\limits_{x\to x_t}\tau_{2j-1}(x)=\lim\limits_{x\to x_t}\tau_{2j}(x)$;
	
	\item $\lim\limits_{x\to+\infty}\tau_1(x)=0$.
\end{enumerate}
\begin{proof}
	Properties (i)--(iv) follow from construction and from continuous dependence (see Figure \ref{fig31} if necessary), while (v) follows as in the proof of Theorem \ref{th31} (observe that, for $x\in D_3$, $\tau_1(x)$ coincides with $\tau_s(x)$ introduced in \eqref{eq35}).
\end{proof}
\end{proposition}

In view of Proposition \ref{pr32}(iv), if we extend the time maps by continuity by setting $\tau_j(x_t):=\lim_{x\to x_t}\tau_j(x)$, we have that
\begin{equation}
	\label{eq37}
	\tau_{2j-1}(x_t)=\tau_{2j}(x_t) \quad \text{ for every $j\in\N^*$.}
	\end{equation}
 All the properties of the time maps that we have established in Proposition \ref{pr32} have been represented in Figure \ref{fig32}.

\begin{figure}[ht]
	\begin{center}
		\includegraphics[scale=0.62]{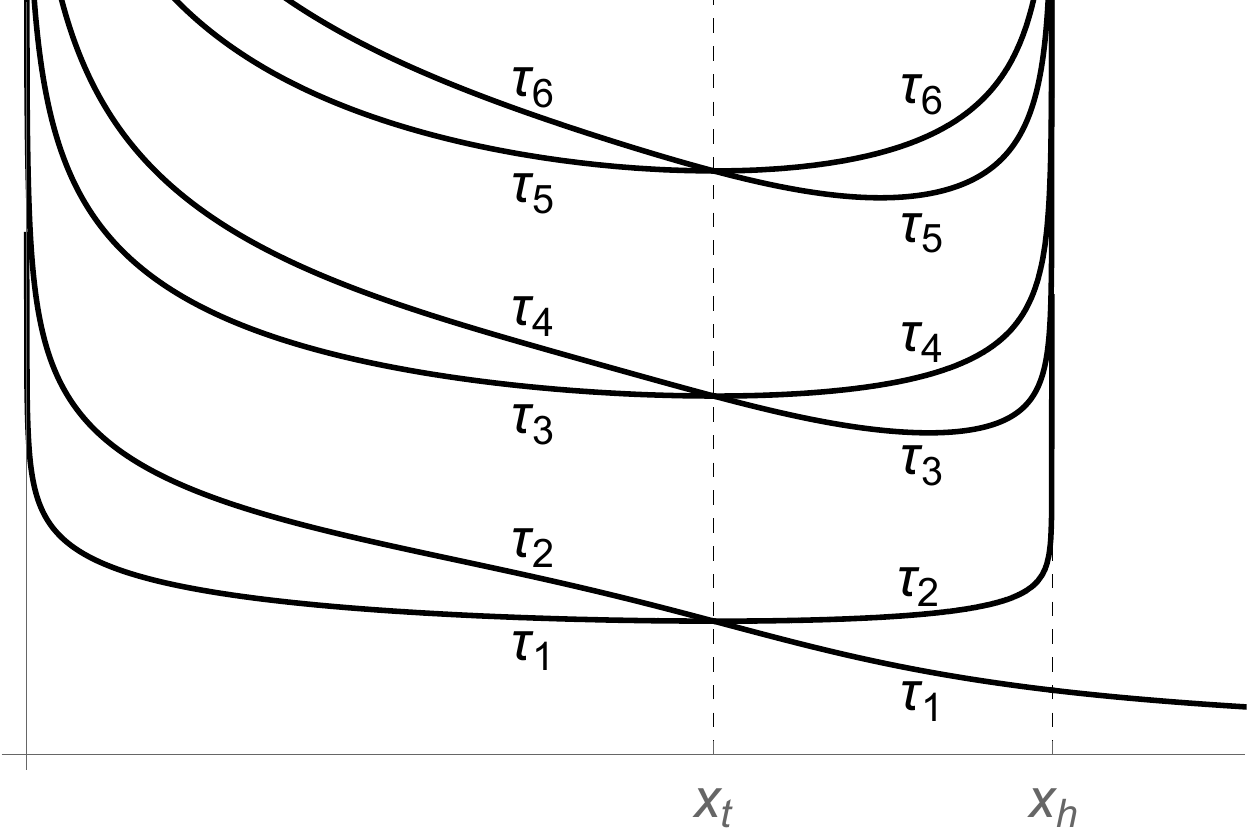} \\
		\caption{Graphs of the time maps $\tau_j(x)$.} \label{fig32}
	\end{center}
\end{figure}

\setcounter{equation}{0}
\section{Singular perturbation from $\a=0$ and high multiplicity results}
\label{section4}
In this section we obtain our high multiplicity results, first of all for the purely superlinear problem corresponding to the case $\a=0$ and later, by performing a singular perturbation, we obtain an analogous result for $\a\sim 0$.

\begin{theorem}
	\label{th41}
For every $n\in\N$, denote $\l_n:=-\frac{(n\pi)^2}{p-1}$. Then, if $\l\in[\l_{n+1},\l_n)$, the purely superlinear problem \eqref{eq11} with $\a=0$ has exactly $2n+1$ solutions.

\begin{proof} Consider the Cauchy problem
	\begin{equation}
	\label{eq41}
	\left\{ \begin{array}{l}
	-u''=\l u+bu^p \quad \hbox{for}\;\; t\in(0,1),\cr
	u'(0)=0, \qquad u(0)=x.\end{array}\right.
	\end{equation}
	A solution of \eqref{eq41} is a solution of Problem \eqref{eq11} with $\a=0$ if and only if, after a certain number of half-laps around the equilibrium $(\O,0)$ in the phase plane, it arrives back on the $u$-axis (see Figure \ref{fig31}, now without considering the curves $\G_0$ and $\G_1$, since $\a=0$). For this reason, for $x\in(0,u_h)\setminus\{\O\}$, we define
\begin{equation*}
T_1(x)=\sign(x_1(x)-x)\int_x^{x_1(x)}\frac{du}{\sqrt{-\l(u^2-x^2)-\frac{2b}{p+1}(u^{p+1}-x^{p+1})}}, 
\end{equation*}
where, by denoting by $M(x)$ and $m(x)$ the greatest and the smallest abscissa (respectively) of the orbit in the phase plane which passes through $(x,0)$, we have set
\begin{equation*}
x_1(x)=\begin{cases}
M(x) & \text{if $x<\O$,} \\m(x) & \text{if $x>\O$.}
\end{cases}
\end{equation*}
$T_1(x)$ is a differentiable function that measures the time needed for the solution of \eqref{eq41} to reach the $u$-axis again for the first time. Similarly, $T_n(x):=nT_1(x)$, $n\in\N^*$, measures the time needed for the solution of Problem \eqref{eq41} to reach the $u$-axis for the $n$th time. The graphs of these time maps have been represented in Figure \ref{fig4.1}.

	\begin{figure}[ht]
	\begin{center}
		\begin{tabular}{c}
			\includegraphics[scale=0.60]{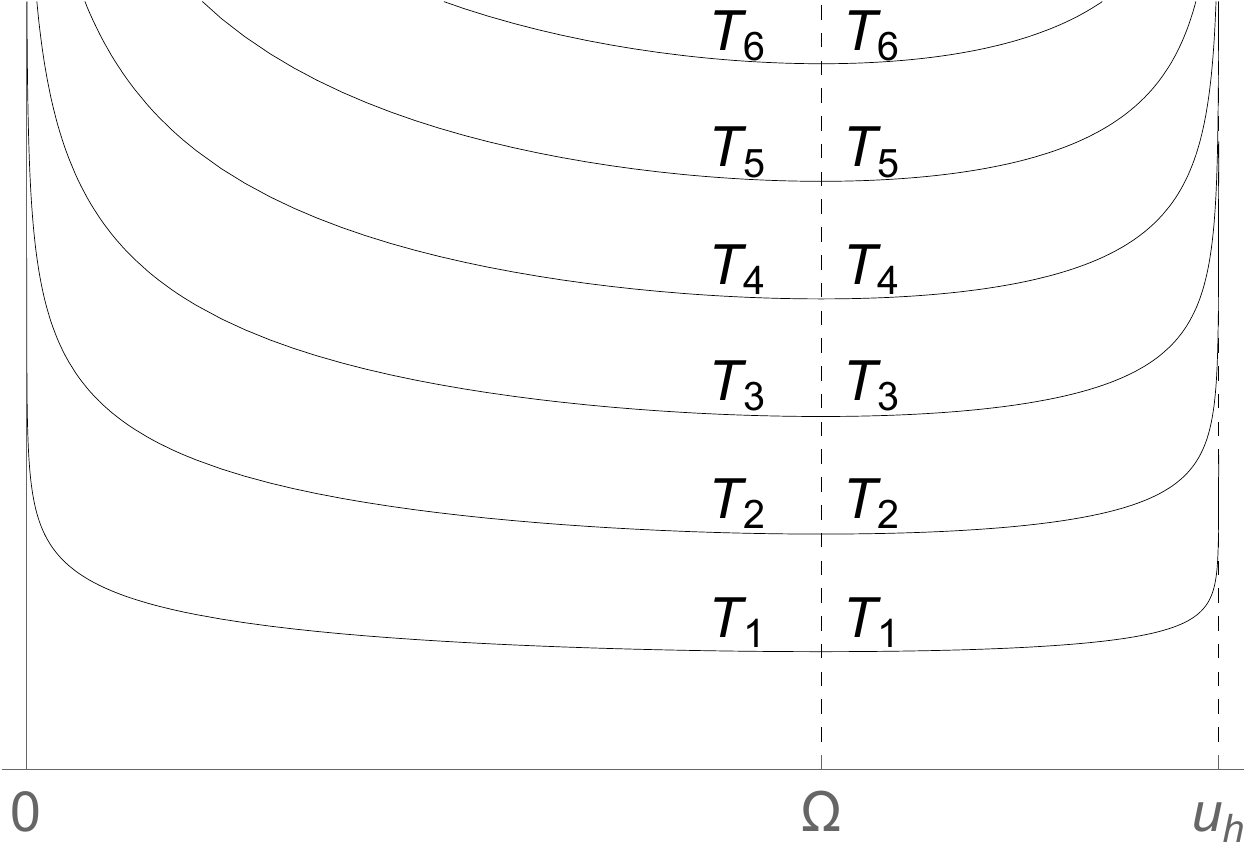} 
		\end{tabular} \\
		\caption{Time maps of Problem \eqref{eq41}.
		} \label{fig4.1}
	\end{center}
\end{figure}

By continuous dependence we have
\begin{equation*}
\label{eq42}
\lim_{x\da0}T_n(x)=+\infty=\lim_{x\ua u_h}T_n(x),
\end{equation*}
since we approach the homoclinic $\g_h$ while, by linearizing around $\O$, $\lim_{x\to\O}T_n(x)=\frac{n\pi}{\sqrt{\l(1-p)}}$, and we can extend $T_n$ by continuity by setting
\begin{equation}
\label{eq43}
T_n(\O)=\frac{n\pi}{\sqrt{\l(1-p)}}.
\end{equation}
By using \cite[Theorem 1 and remarks on page 102]{S} and that, from the symmetry of the problem, $T_n(x)=T_n(x_1(x))$, we obtain that
\begin{equation}
\label{monotoniaT}
T_n(x) \text{ is decreasing for $0<x<\O$ and increasing for $\O<x<u_h$}.
\end{equation}
Assume now that, $\l\in[\l_{n+1},\l_n)$ for some $n\in\N^*$. By \eqref{eq43}, this is equivalent to
$T_n(\O)<1\leq T_{n+1}(\O)$,
thus,  the previous analysis guarantees that the equation $T_j(x)=1$ has exactly $2$ solutions,  $x_j^-\in(0,\O)$ and $x_j^+\in(\O,u_h)$, for every $j=1,\ldots,n$, while it has no solution for $j>n$. Moreover, since  for every $n\in\N^*$ $T_{n+1}(x)>T_n(x)$, $x_{j_1}^-< x_{j_2}^-$ and $x_{j_1}^+> x_{j_2}^+$ if $1\leq j_1< j_2\leq n$, implying that the corresponding solutions of Problem \eqref{eq11} with $\a=0$ are different.

Observing in addition that Problem \eqref{eq11} with $\a=0$ always admits the constant solution $u\equiv \O$, we have proved that, when $\l\in[\l_{n+1},\l_n)$ for some $n\in\N^*$, the problem has exactly $2n+1$ solutions. Finally, when $\l\in[\l_1,\l_0)$, $T_1(\O)\geq 1$, thus the unique solution is $u\equiv\O$, which concludes the proof.
\end{proof}
\end{theorem}

In order to obtain the analogous multiplicity result for $\a\sim 0$, we study the behavior of the time maps $\tau_j$ introduced in Section \ref{section3} as $\a\da 0$. To this end, we preliminarily need the following geometrical result related to the behavior of $x_t(\a)$ (here and in the following we explicitly write the dependence on $\a$ of the quantities introduced in the previous sections).

\begin{proposition}
	\label{pr42}
	$x_t(\a)$ is decreasing for $\a\sim 0$ and satisfies
	\begin{equation}
		\label{asymptotics}
		x_t(\a)=\O+o(\a) \text{ as $\a\da 0$}.
		\end{equation}
		\begin{proof}
			As remarked above (see \eqref{eq32}), the equation of the upper half-orbit in the phase plane $(u,v)$ through a point whose coordinates are $(x,y)$, with $y>0$, is given by
			\begin{equation*}
				v(u,x,y)=\sqrt{y^2-\l\left(u^2-x^2\right)-\frac{2b}{p+1}\left(u^{p+1}-x^{p+1}\right)},
				\end{equation*}
				therefore
				\begin{equation}
					\label{tangenttoorbit}
					\frac{\p v}{\p u}(x,x,y)=\frac{-\l x-bx^p}{y},
					\end{equation}
				which is positive if and only if $x\in(0,\O)$.	As $\p y/\p x>0$ (due to Theorem \ref{th22}\eqref{th22iii}), \eqref{tangenttoorbit} implies that $x_t(\a)<\O$, since the tangent to $\G_0(\a)$ and to the orbit through the point $(x_t(\a),y(x_t(\a),\a))$ must be the same, in particular its slope must be positive. 
					
					We fix now $0<\a_2<\a_1$, $\a_1\sim 0$ and set $\ov x:=x_t(\a_1)$.
					From \eqref{dydalpha}, we deduce $y(\ov x,\a_2)<y(\ov x,\a_1)$, while \eqref{dy'dalpha} implies $\frac{\p y}{\p x}(\ov x,\a_2)<\frac{\p y}{\p x}(\ov x,\a_1)$. This, together with \eqref{tangenttoorbit}, gives
					\[\frac{\p y}{\p x}\left(\ov x,\a_2\right)<\frac{\p y}{\p x}\left(\ov x,\a_1\right)=\frac{\p v}{\p u}\left(\ov x ,\ov x, y(\ov x,\a_1)\right)<\frac{\p v}{\p u}\left(\ov x ,\ov x, y(\ov x,\a_2)\right).\] Thus, the orbit through $(\ov x,y(\ov x,\a_2))$ is secant to $\G_0(\a_2)$ and the convexity of $\G_0(\a_2)$ (see  Proposition \ref{pr25}(iii)) implies $x_t(\a_1)<x_t(\a_2)$. As a consequence, the limit of $x_t(\a)$ as $\a\da 0$ exists and is positive. Let us denote it by $l$.
					
					By imposing that the derivative of the orbit of the superlinear problem, given by \eqref{tangenttoorbit}, coincides with $\frac{\p y}{\p x}(x_t(\a),\a)$, we obtain the following implicit relation which characterizes $x_t(\a)$:
					\begin{equation}
						\label{implicitx_t}
						\frac{\p y}{\p x}(x_t(\a),\a)y(x_t(\a),\a)=-\l x_t(\a)-b x_t^{p}(\a).
						\end{equation}
						By taking the limit in \eqref{implicitx_t},  Proposition \ref{pr25}(i)--(ii) implies that $l$ satisfies $-\l l-b l^p=0$, whose unique positive solution is $l=\O$. This proves the zeroth order term in the expansion \eqref{asymptotics}. For the first order term, we differentiate \eqref{implicitx_t} with respect to $\a$, obtaining
						\[\left(\frac{\p^2y}{\p x^2}x_t'(\a)+\frac{\p^2y}{\p\a\p x}\right)y(x_t(\a),\a)+\frac{\p y}{\p x}(x_t(\a),\a)\left(\frac{\p y}{\p x}x_t'(\a)+\frac{\p y}{\p \a}\right)=\left[-\l-bpx_t(\a)^{p-1}\right]x_t'(\a).\]
						As $\a\da 0$, the left-hand side tends to 0 thanks to Proposition \ref{pr25}(i)--(iii), while the first factor in the right-hand side tends to $\l(p-1)<0$, giving $\lim_{\a\da 0}x_t'(\a)=0$, as we wanted.
			\end{proof}
	\end{proposition}
	
	As a consequence of the previous proposition, in order to determine $\lim_{\a\da 0}\tau_1(x_t(\a))$ we need to perform a singular perturbation result. This is the content of the following proposition.
\begin{proposition}
	\label{pr43}
The following holds true:
\begin{equation*}
\lim_{\a\da 0}\tau_1(x_t(\a),\a)=\frac{\pi}{\sqrt{\l(1-p)}}.
\end{equation*}
\begin{proof}
The Taylor expansion of the nonlinearity in \eqref{eq31} around $\O$ gives \[-(u(t)-\O)''=\l(1-p)(u(t)-\O)+\sum_{j\geq 2}\tilde c_j(u(t)-\O)^j\] for some coefficients $\tilde c_j$, thus the expression of the orbit through $(x_t(\a),y(x_t(\a),\a))$ given by \eqref{eqenergy} reads
\begin{multline*}
	v^2+\l(1-p)(u-\O)^2+\sum_{j\geq 3}c_j(u-\O)^j= \\ =y^2(x_t(\a),\a)+\l(1-p)\left(x_t(\a)-\O\right)^2+\sum_{j\geq 3}c_j(x_t(\a)-\O)^j,
	\end{multline*}
where $c_j=\frac{2}{j}\tilde c_{j-1}$, and the time to connect $\G_0$ to $\G_1$ for the first time along this orbit is
\begin{multline*}
\tau_1(x_t(\a),\a)=2\int_{x_t(\a)-\O}^{M(\a)-\O}\Biggl(y^2(x_t(\a),\a)+\l(1-p)\left(x_t(\a)-\O\right)^2+\Biggr. \\
+\left.\sum_{j\geq 3}c_j(x_t(\a)-\O)^j-\l(1-p)\t^2-\sum_{j\geq 3}c_j\t^j\right)^{-1/2}d\t,
\end{multline*}
where, in order to shorten the notation, we have set $M(\a)=M(x_t(\a),\a)$.
We can split the integral in two summands: $I_1$ for $\t\in(x_t(\a)-\O,0)$ and $I_2$ for  $\t\in(0,M(\a)-\O)$. Regarding $I_1$,  by using the estimate
\begin{equation*}
\biggl\lvert\sum_{j\geq 3}c_j\t^j\biggr\rvert\leq\sum_{j\geq 3}|c_j||x_t(\a)-\O|^j,
\end{equation*}
we have
\begin{equation*}
\int_{x_t(\a)-\O}^{0}\frac{d\t}{\sqrt{h^+_1(\a)-\l(1-p)\t^2}}\leq I_1(\a)\leq\int_{x_t(\a)-\O}^{0}\frac{d\t}{\sqrt{h^-_1(\a)-\l(1-p)\t^2}},
\end{equation*}
where
\begin{equation*} h^{\pm}_1(\a):=y^2(x_t(\a),\a)+\l(1-p)\left(x_t(\a)-\O\right)^2+\sum_{j\geq 3}c_j(x_t(\a)-\O)^j\pm\sum_{j\geq 3}|c_j||x_t(\a)-\O|^j.
	\end{equation*}
	Now,
\begin{equation*}
\int_{x_t(\a)-\O}^{0}\frac{d\t}{\sqrt{h^\pm_1(\a)-\l(1-p)\t^2}}=\frac{1}{\sqrt{\l(1-p)}}\int_{-\sqrt{\l(1-p)\frac{(x_t(\a)-\O)^2}{h^\pm_1(\a)}}}^{0}\frac{d\xi}{\sqrt{1-\xi^2}}.
\end{equation*}

Recalling \eqref{asymptotics}, \eqref{dydalpha} and Proposition \ref{pr25}(ii), we have
\[\lim_{\a\da0}\frac{y(x_t(\a),\a)}{x_t(\a)-\O}=\lim_{\a\da0}\frac{\frac{\p y}{\p x}x_t'(\a)+\frac{\p y}{\p \a}}{x_t'(\a)}=+\infty,
\]
thus
\begin{equation*}
	\frac{h^{\pm}_1(\a)}{(x_t(\a)-\O)^2}=\frac{y(x_t(\a),\a)^2}{(x_t(\a)-\O)^2}+\l(1-p)+o\left(1\right)\to +\infty
	\end{equation*}
	as $\a\to 0$, and $\lim_{\a\da 0}I_1(\a)=0$.

Passing to $I_2$, since $\t\in(0,M(\a)-\O)$, we have 
\begin{equation*}
	\bigl\lvert\sum_{j\geq 3}c_j\t^j\bigr\rvert\leq\sum_{j\geq 3}|c_j|(M(\a)-\O)^j,
\end{equation*}
and we get the following estimate
\begin{equation*}
	\int_{0}^{M(\a)-\O}\frac{d\t}{\sqrt{h^+_2(\a)-\l(1-p)\t^2}}\leq I_2(\a)\leq\int_{0}^{M(\a)-\O}\frac{d\t}{\sqrt{h^-_2(\a)-\l(1-p)\t^2}},
\end{equation*}
where
\begin{equation*}
	h^{\pm}_2(\a):=y^2(x_t(\a),\a)+\l(1-p)\left(x_t(\a)-\O\right)^2+\sum_{j\geq 3}c_j(x_t(\a)-\O)^j\pm\sum_{j\geq 3}|c_j|(M(\a)-\O)^j.
	\end{equation*}
	Now,
\begin{equation*}
	\int_{0}^{M(\a)-\O}\frac{d\t}{\sqrt{h^\pm_2(\a)-\l(1-p)\t^2}}=\frac{1}{\sqrt{\l(1-p)}}\int_{0}^{\sqrt{\l(1-p)\frac{(M(\a)-\O)^2}{h^\pm_2(\a)}}}\frac{d\xi}{\sqrt{1-\xi^2}}.
\end{equation*}
To compute the limit as $\a\da 0$ we  need preliminarily to establish the following asymptotic expansion for $M(\a)$:
\begin{equation}
	\label{asymptoticsM}
	M(\a)=\O+\frac{1}{\sqrt{\l(1-p)}}\frac{\p y}{\p\a}(\O,0)\a+o(\a) \text{ as $\a\sim 0$.}
\end{equation}
Indeed, $M(\a)>\O$ is implicitly defined through
\begin{equation}
	\label{implicitM}
	y^2(x_t(\a),\a)+\l x^2_t(\a)+\frac{2b}{p+1}x^{p+1}_t(\a)=\l M^2(\a)+\frac{2b}{p+1}M^{p+1}(\a),
\end{equation}
which, by using Proposition \ref{pr25}(i) and \eqref{asymptotics}, gives that $L:=\lim_{\a\da 0}M(\a)$ satisfies $\l L^2+\frac{2b}{p+1}L^{p+1}=\l\O^2+\frac{2b}{p+1}\O^{p+1}$, whose unique nonnegative solution is  $L=\O$. Moreover, differentiating \eqref{implicitM} twice with respect to $\a$ (differentiating once and taking $\a\to 0$ gives nothing but $0=0$) leads to
\begin{multline*}
	\left(\frac{\p y}{\p x}x_t'(\a)+\frac{\p y}{\p\a}\right)^2+y(x_t(\a),\a)\frac{d}{d\a}\left(\frac{\p y}{\p x}x_t'(\a)+\frac{\p y}{\p\a}\right)+ \\ + x_t'^2(\a)\left(\l+bpx_t^{p-1}(\a)\right)
	+x_t(\a)x_t''(\a)\left(\l+bx_t^{p-1}(\a)\right)= \\ =M'^2(\a)\left(\l+bpM^{p-1}(\a)\right)+
	M(\a)M''(\a)\left(\l+bM^{p-1}(\a)\right),
\end{multline*}
which, by taking $\a\to 0$, reduces to $\left(\frac{\p y}{\p\a}(\O,0)\right)^2=M'^2(0)\l(1-p)$, completing the proof of \eqref{asymptoticsM}.

Recalling \eqref{asymptotics}, \eqref{asymptoticsM}, \eqref{dydalpha} and Proposition \ref{pr25}(ii), we have
\[\lim_{\a\da0}\frac{y(x_t(\a),\a)}{M(\a)-\O}=\lim_{\a\da0}\frac{\frac{\p y}{\p x}x_t'(\a)+\frac{\p y}{\p \a}}{M'(\a)}=\sqrt{\l(1-p)} \quad \text{and} \quad \lim_{\a\da0}\frac{x_t(\a)-\O}{M(\a)-\O}=0,
\]
thus
\begin{multline*}
	\frac{h^{\pm}_2(\a)}{(M(\a)-\O)^2}=\frac{y(x_t(\a),\a)^2}{(M(\a)-\O)^2}+\l(1-p)\frac{(x_t(\a)-\O)^2}{(M(\a)-\O)^2}+ \\
	+ \sum_{j\geq 3}c_j\frac{(x_t(\a)-\O)^2}{(M(\a)-\O)^2}(x_t(\a)-\O)^{j-2}\pm\sum_{j\geq 3}|c_j|(M(\a)-\O)^{j-2}\to \l(1-p)
\end{multline*}
as $\a\to 0$, and $\lim_{\a\da 0}I_2(\a)=\frac{\arcsin(1)}{\sqrt{\l(1-p)}}=\frac{\pi}{2\sqrt{\l(1-p)}}$, concluding the proof.
\end{proof}
\end{proposition}

\begin{remark}
A posteriori, we observe that the limiting time to connect the tangency point on $\G_0$ to the one on $\G_1$ through the flow of \eqref{eq31}, which has been computed in Proposition \ref{pr43}, coincides with the time of the linearized problem around $(\O,0)$. Indeed, the latter is proportional to angle spanned in the limit, which, by elementary geometrical considerations, is given by
\[\theta=2\left(\pi-\lim_{\a\to 0}\arctan\frac{y(x_t(\a),\a)}{\O-x_t(\a)}\right). \]
To compute the limit in the last expression, we divide \eqref{implicitx_t} by $\O-x_t(\a)$, obtaining
\[\frac{\p y}{\p x}(x_t(\a),\a)\frac{y(x_t(\a),\a)}{\O-x_t(\a)}=\frac{-\l x_t(\a)-b x_t^{p}(\a)}{\O-x_t(\a)}.
\]
Passing to the limit in this expression, we observe, by using de l'H\^{o}pital's rule, that the right-hand side converges to $\l(1-p)>0$ while the first factor in the left-hand side converges to $0$ by Proposition \ref{pr25}(ii). Thus $\frac{y(x_t(\a),\a)}{\O-x_t(\a)}$ converges to $+\infty$, $\theta=\pi$, and the limiting time equals half of a period of the linearization of \eqref{eq31} around $(\O,0)$, i.e. $\frac{\pi}{\sqrt{\l(1-p)}}$.
\end{remark}

We are now able to obtain a high multiplicity result for $\a\sim 0$ by performing a singular perturbation from the case $\a=0$.

\begin{theorem}
	\label{th45}
For every $n\in\N$, set $\l_n:=-\frac{(n\pi)^2}{p-1}$ as in Theorem \ref{th41}. Then, if $\l\in[\l_{n+1},\l_n)$, there exists $\a^*=\a^*(\l)$ such that Problem \eqref{eq11}  has at least $2n+1$ solutions for every $\a\in(0,\a^*)$, each of which converges as $\a\da 0$, uniformly for $t\in[0,1]$, to one of the $2n+1$ solutions that Problem \eqref{eq11} admits for $\a=0$.

\begin{proof}
	First of all, observe that, if we set $\tilde D_1=(0,\O)$, $\tilde D_2=(\O,u_h)$ and $\tilde D_3=[u_h,+\infty)$, Propositions  \ref{pr42} and  \ref{pr25}(i) guarantee that the domains $D_1(\a)$, $D_2(\a)$ and $D_3(\a)$ of the time maps introduced in Section \ref{section3} converge to $\tilde D_1$, $\tilde D_2$ and $\tilde D_3$, respectively, as $\a\to 0$.
	
	In addition, by the differentiable dependence theorem for differential equations, which holds uniformly in compact sets not containing any equilibria, we have that, for $\a\to 0$,
	\begin{align}
	\label{equnif1}
	\tau_j&\to T_j & &\text{ uniformly in compact sets of $\tilde D_1$, for every $j\in\N^*$,} \\
		\label{equnif2}
	\tau_{j+1}&\to T_j & &\text{ uniformly in compact sets of $\tilde D_2$, for every $j\in\N^*$,} \\
	\label{equnif3}
	\tau_{1}&\to 0 & &\text{ uniformly in compact sets of $\tilde D_2\cup\tilde D_3$,}
	\end{align}
	where $T_j$ are the time maps of the case $\a=0$ introduced in the proof of Theorem \ref{th41} (Figure \ref{fig4.2} shows such convergences). In addition, analogous convergence results hold for the derivatives of the time maps with respect to $x$.
	\begin{figure}[ht]
		\begin{center}
			\begin{tabular}{c}
			\includegraphics[scale=0.60]{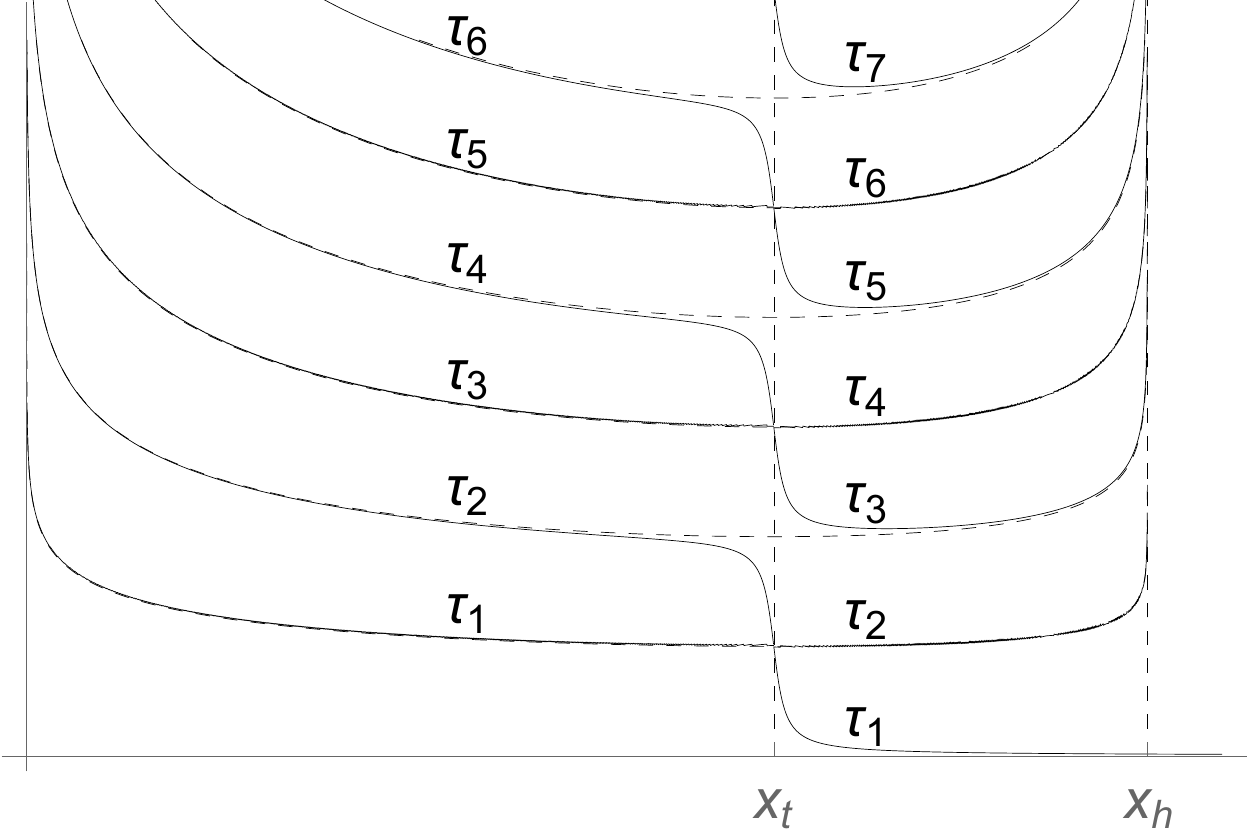}
			\end{tabular} \\
			\caption{Time maps of Problem \eqref{eq11} for $\a\sim 0$. We have represented the time maps $\tau_j$ with continuous lines, while the dashed lines represent the time maps $T_j$ of Figure \ref{fig4.1}.} \label{fig4.2}
		\end{center}
	\end{figure}

	Combining Propositions \ref{pr32}(i),(iv) and \ref{pr43}, we obtain, for every $j\in2\N+1$,
	\begin{equation}
	\label{eq411}
	\lim_{\a\da 0}\tau_j(x_t(\a),\a)=\frac{j\pi}{\sqrt{\l(1-p)}}=\lim_{\a\da 0}\tau_{j+1}(x_t(\a),\a).
	\end{equation}
	Assume now that $\l\in[\l_{n+1},\l_n)$ for some $n\in\N^*$, which is equivalent to
	\begin{equation}
	\label{eq412}
	\frac{n\pi}{\sqrt{\l(1-p)}}<1\leq\frac{(n+1)\pi}{\sqrt{\l(1-p)}}. 
	\end{equation}
	From the proof of Theorem \ref{th41}, there exist $n$ points $\{x_j^-\}_{j=1}^n$ such that $0<x_1^-<\ldots<x_j^-<\ldots<x_n^-<\O$ and
	\begin{equation}
	\label{eq413}
	\{x_1^-,\ldots,x_n^-\}=\bigcup_{j=1}^{n}T_j^{-1}(1)\cap \tilde D_1.
	\end{equation}
	Take $\e> 0$, $\e\sim 0$ such that $(x_1^--\e,x_n^-+\e)\subset \tilde D_1$. Since $1-2\a\to 1$ as $\a\to 0$, \eqref{equnif1}, \eqref{eq413} and the convergence of the derivatives of the time maps with respect to $x$ guarantee that
		\begin{equation}
					\label{solD1}
	\bigcup_{j=1}^{n}\tau_j^{-1}(1-2\a)\cap (x_1^--\e,x_n^-+\e)
	\end{equation}
	has at least $n$ elements for $\a\sim 0$, corresponding to $n$ solutions of Problem \eqref{eq11}, each of which converges to the corresponding one of the case $\a=0$, thanks again to  \eqref{equnif1} and \eqref{eq413}, giving the corresponding uniform convergence for the solutions of Problem \eqref{eq11}.
	
	Similarly, there exist $n$ points $\{x_j^+\}_{j=1}^n$ such that $\O<x_n^+<\ldots<x_j^+<\ldots<x_1^+<u_h$ and
	\begin{equation}
	\label{eq414}
	\{x_1^+,\ldots,x_n^+\}=\bigcup_{j=1}^{n}T_j^{-1}(1)\cap \tilde D_2.
	\end{equation}
	By taking $\e> 0$, $\e\sim 0$ such that $(x_n^+-\e,x_1^++\e)\subset \tilde D_2$, \eqref{equnif2}, \eqref{eq414} and the convergence of the derivatives of the time maps with respect to $x$ now guarantee that
	\begin{equation}
		\label{solD2}
	\bigcup_{j=2}^{n+1}\tau_j^{-1}(1-2\a)\cap (x_n^+-\e,x_1^++\e)
	\end{equation}
	has at least $n$ elements for $\a\sim 0$, providing other $n$ solutions of Problem \eqref{eq11} which converge, as $\a\da 0$, to the corresponding ones of the case $\a=0$.
	
	In summary we have proved that, if $\l\in[\l_{n+1},\l_n)$ for some $n\in\N^*$, Problem \eqref{eq11} admits at least $2n$ solutions in a right neighborhood of $\a=0$. In order to prove the existence of a $(2n+1)$th solution if $\l\in[\l_{n+1},\l_n)$, $n\in\N$, we proceed as follows: set $n=2k+r$ with $k\in\N$ and $r\in\{0,1\}$ and define the function
	\begin{equation}
		\label{eq416}
					\T_k(x,\a):=
		\begin{cases}
			\tau_{2k+2}(x,\a) & \text{ if $0<x< x_t$,} \\
						\tau_{2k+1}(x,\a) & \text{ if $x_t\leq x< x_h$,}
			\end{cases}
		\end{equation}
		which is continue by \eqref{eq37}. Moreover, reducing $\e$ if necessary so that $\O+\e<u_h$ and, when $n\in\N^*$,
		\begin{equation}
			\label{condeps}
			x_n^-+\e<\O-\e<\O+\e<x_n^+-\e,
			\end{equation}
		from \eqref{equnif2} and \eqref{equnif3} we have that
		\begin{equation*}
			\inf_{(\O,\O+\e)}\T_k(\cdot,\a)\to \frac{2k\pi}{\sqrt{\l(1-p)}}\leq \frac{n\pi}{\sqrt{\l(1-p)}}
			\end{equation*}
		as $\a\to 0$, thus \eqref{eq412} gives the existence of $\ov x\in (\O,\O+\e)$ such that $\T_k(\ov x,\a)<1-2\a$ for $\a\sim 0$. On the other hand, Proposition \ref{pr32}(iii) gives that $\lim_{x\da0}\T_k(x,\a)=+\infty$, thus by continuity we obtain, for $\a\sim 0$, the existence of $x(\a)<\ov x<\O+\e$ such that $\T_k(x(\a),\a)=1-2\a$ which, by construction, provides us with a solution of \eqref{eq11}.
		
		To show that this solution is different from the ones found above when $n\in\N^*$, observe that, if $x(\a)\geq x_t(\a)$, \eqref{asymptotics} guarantees that $x(\a)\in(\O-\e,\O+\e)$ for $\a\sim 0$ and thus, this solution cannot coincide with any of the previous solutions, due to \eqref{solD1}, \eqref{solD2} and \eqref{condeps}. If instead $x(\a)< x_t(\a)$, the definition of $\T_k$ and Proposition \ref{pr32}(ii) give that $x(\a)$ satisfies
		\begin{equation}
			\label{eq419}
			1-2\a=\tau_{2k+2}(x(\a),\a)>\tau_{n}(x(\a),\a),
			\end{equation}
		thus $x(\a)$ does not belong to the set \eqref{solD1}. It cannot belong to \eqref{solD2} either, because $x(\a)<\O+\e<x_n^+-\e$.
		
		It only remains to show that this solution of Problem \eqref{eq11} converges to the constant $\O$ as $\a\to 0$. Firstly, we prove that $x(\a)\to\O$. We distinguish again the cases $x(\a)\geq x_t(\a)$ and $x(\a)<x_t(\a)<\O$. In the former, as already remarked above, \eqref{asymptotics} guarantees that $x(\a)\in(\O-\e,\O+\e)$ for $\a\sim 0$ and the claim follows since $\e$ can be arbitrarily small. In the latter, assume by contradiction that $\tilde x:=\limsup_{\a\da 0} x(\a)<\O$. By taking the limsup as $\a\da 0$ in \eqref{eq419}, we obtain from \eqref{equnif1} that $\tilde x$ solves $T_{2k+2}(\cdot)=1$. Nevertheless, in the considered range of $\l$, this equation only has a solution when $n$ is odd and $\l=\l_{n+1}$, and, in such a case,  \eqref{monotoniaT} implies that the solution is $\O$, contradicting $\tilde x<\O$. An analogous argument with the liminf concludes the proof of the claim.
		
		The uniform convergence in $[0,1]$ of the solution of \eqref{eq11} to $\O$ finally follows by combining that $x(\a)\to\O$ with Proposition \ref{pr25}(i) and the fact that the closed orbits of \eqref{eq31} degenerate to $(\O,0)$ when they approach the equilibrium.		
\end{proof}
\end{theorem}

\begin{remark}
\label{re46}
	Contrarily to Theorem \ref{th41}, in Theorem \ref{th45} we are not able to provide exact multiplicity results, since we do not know the global monotonicities of the time maps $\tau_j$, which was instead the case for the $T_j$.
\end{remark}

\setcounter{equation}{0}
\section{Bifurcation diagrams in $\a$ and general multiplicity results}
\label{section5}
In this section we determine the structure of the global bifurcation diagrams of Problem \eqref{eq11} using $\a$ as the main bifurcation parameter. Moreover, we determine some general multiplicity results which complete the ones obtained in Theorem \ref{th45} for $\a\sim 0$ and Theorem \ref{th31} for $\a=0$. They are contained in the following theorem and represented in Figure \ref{fig5}, where we plot the values of $u(\a)$ on the vertical axis.


\begin{theorem}
	\label{th51}
	Assume that $\l\in[\l_{n+1},\l_n)$ for some $n\in\N$. Then:
	\renewcommand{\theenumi}{\roman{enumi}}
		\begin{enumerate}[(i)]
		\item if $n=0$, the minimal bifurcation diagram in $\a$ for Problem \eqref{eq11} consists of a curve starting from $\{\a=0\}$ and bifurcating from $+\infty$ at $\a=1/2$. Such a curve will be referred to as \emph{principal curve}.
		
		In particular, Problem \eqref{eq11} admits at least a solution for every $\a\in(0,1/2)$ (see Figure \ref{fig5}(A));
		
		\item if $n=1$, the minimal bifurcation diagram in $\a$ for Problem \eqref{eq11} consists of one component containing the principal curve with two additional branches that start from $\{\a=0\}$ and merge in a bifurcation point on the principal curve.
		
		In particular, there exists $\a_1\in(0,1/2)$, such that Problem \eqref{eq11} possesses at least 3 solutions for $\a<\a_{1}$,  and at least 1 solution for $\a_1<\a<1/2$ (see Figure \ref{fig5}(B));
		\item if $n=2k+1$, $k\in\N^*$, the minimal bifurcation diagram in $\a$ for Problem \eqref{eq11} consists of $k+1$ components: one, as in (ii), containing the principal curve with two branches bifurcating from it and reaching the axis $\{\a=0\}$, plus $k$ additional bounded components, each formed by four branches that start from the axis $\{\a=0\}$ and three of which merge in a bifurcation point, while (at least) two of them merge in a subcritical turning point.
		
		In particular, there exist $\{\a_{2j+1}\}_{j=0}^k$ such that $0<\a_{2j_2+1}<\a_{2j_1+1}<1/2$ for $0\leq j_1<j_2\leq k$ and Problem \eqref{eq11} possesses at least $4j+3$ solutions for $\a<\a_{2j+1}$, $j\in\{0,\ldots,k\}$ and at least 1 solution for $\a_1<\a<1/2$ (see Figure \ref{fig5}(D));
		\item if $n=2k$, $k\in\N^*$, the minimal bifurcation diagram in $\a$ for Problem \eqref{eq11} consists of $k+1$ components: one, as in (ii), containing the principal curve with two branches bifurcating from it and reaching the axis $\{\a=0\}$, $k-1$ bounded components as in (iii), each formed by four branches that start from $\{\a=0\}$ and form a subcritical turning point and a bifurcation point, and an additional bounded component formed by two branches that start from the axis $\{\a=0\}$ and merge in a subcritical turning point.
		
		In particular, there exist $\a_{2k}\in(0,1/2)$ and $\{\a_{2j+1}\}_{j=0}^{k-1}$, such that $0<\a_{2j_2+1}<\a_{2j_1+1}<1/2$ for $0\leq j_1<j_2\leq k-1$ and Problem \eqref{eq11} possesses at least $4k+1=2n+1$ for $\a<\min\{\a_{2k},\a_1,\ldots,\a_{2k-1}\}$, at least $4j+3$ solutions for $\a<\a_{2j+1}$, $j\in\{0,\ldots,k-1\}$  and at least one solution for $\a_1<\a<1/2$ (see Figure \ref{fig5}(C)).		
	\end{enumerate}

		\begin{figure}[ht]
	\begin{center}
		\begin{tabular}{cc}
			\includegraphics[scale=0.52]{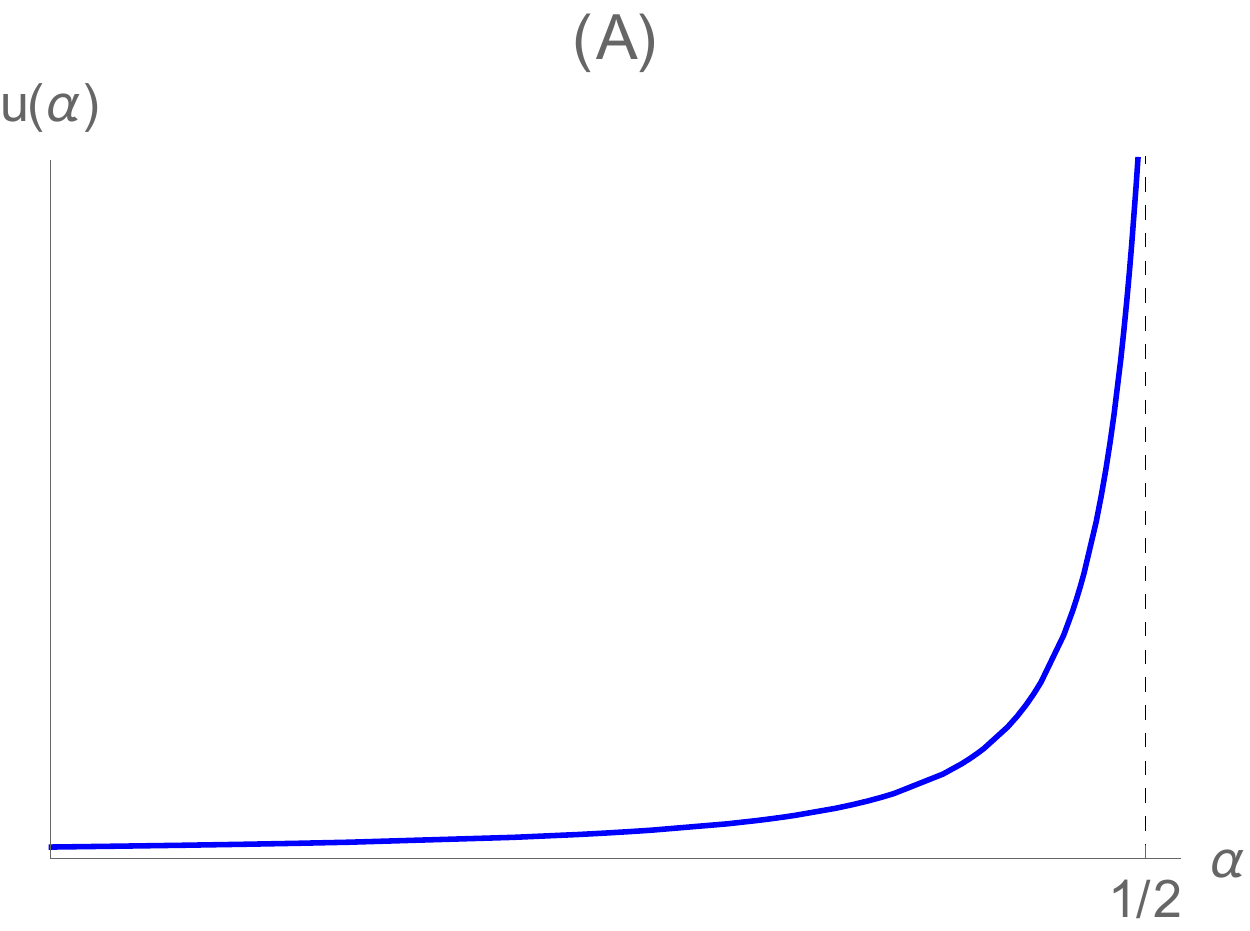} & 			\includegraphics[scale=0.52]{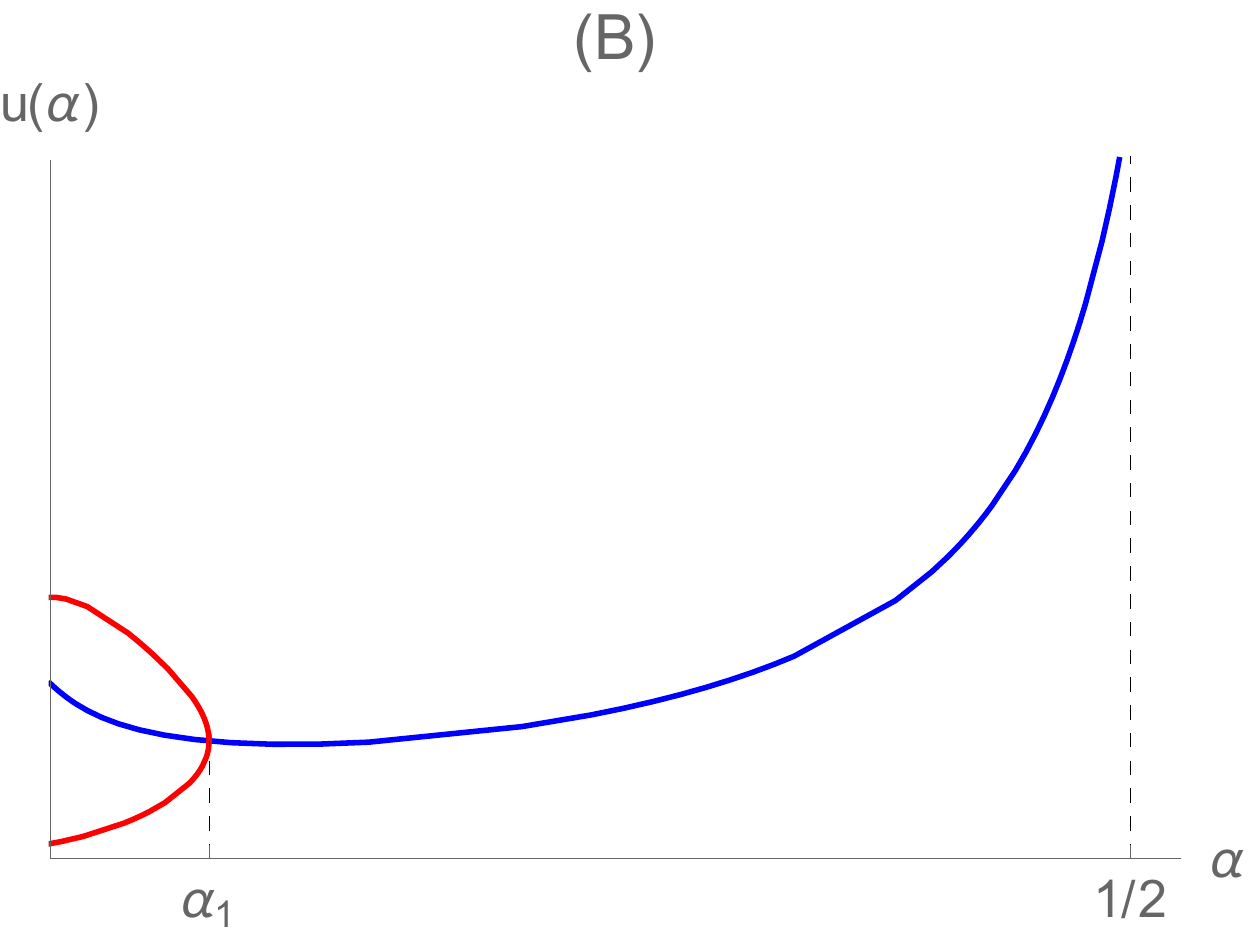} \\
			\includegraphics[scale=0.52]{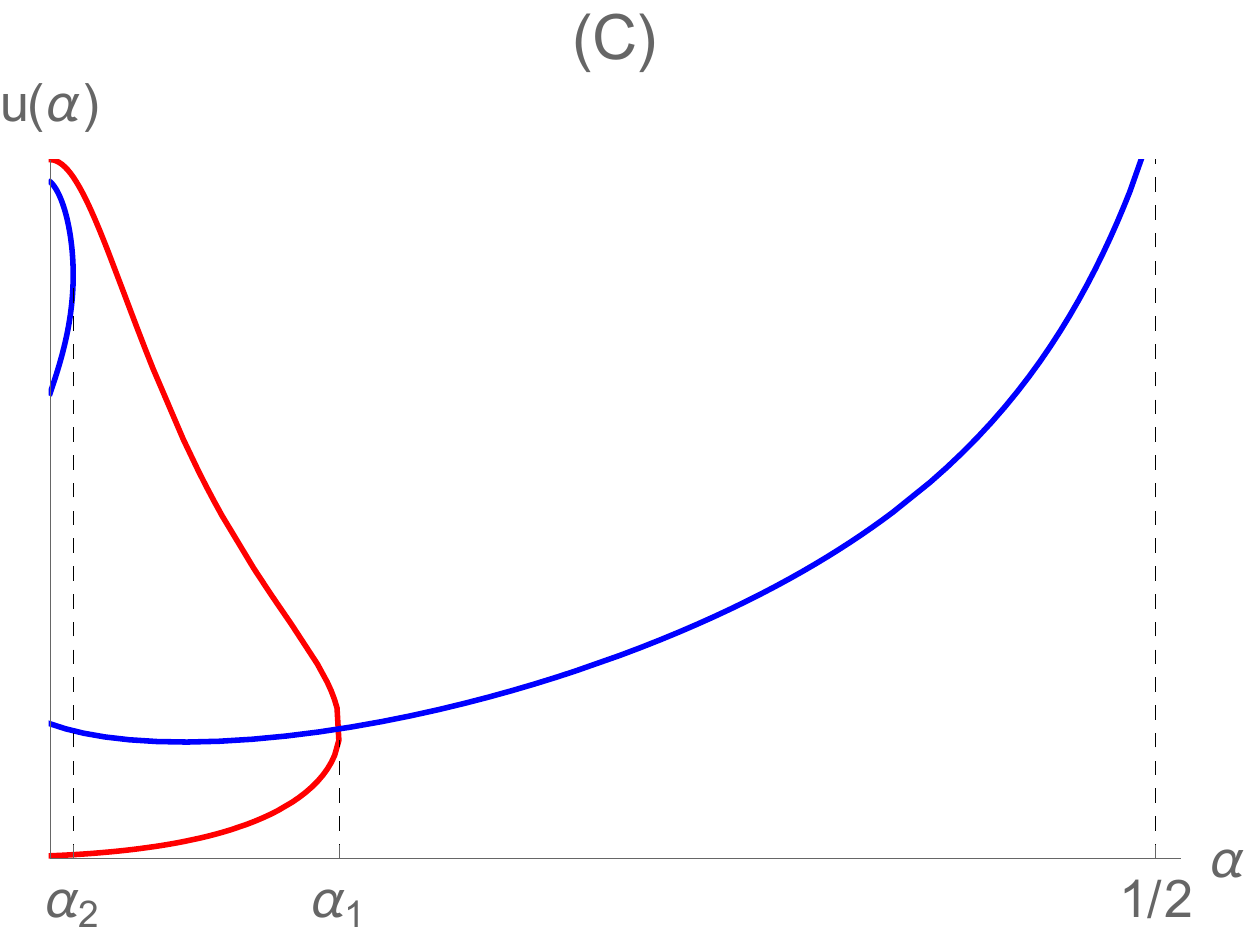} & 			\includegraphics[scale=0.52]{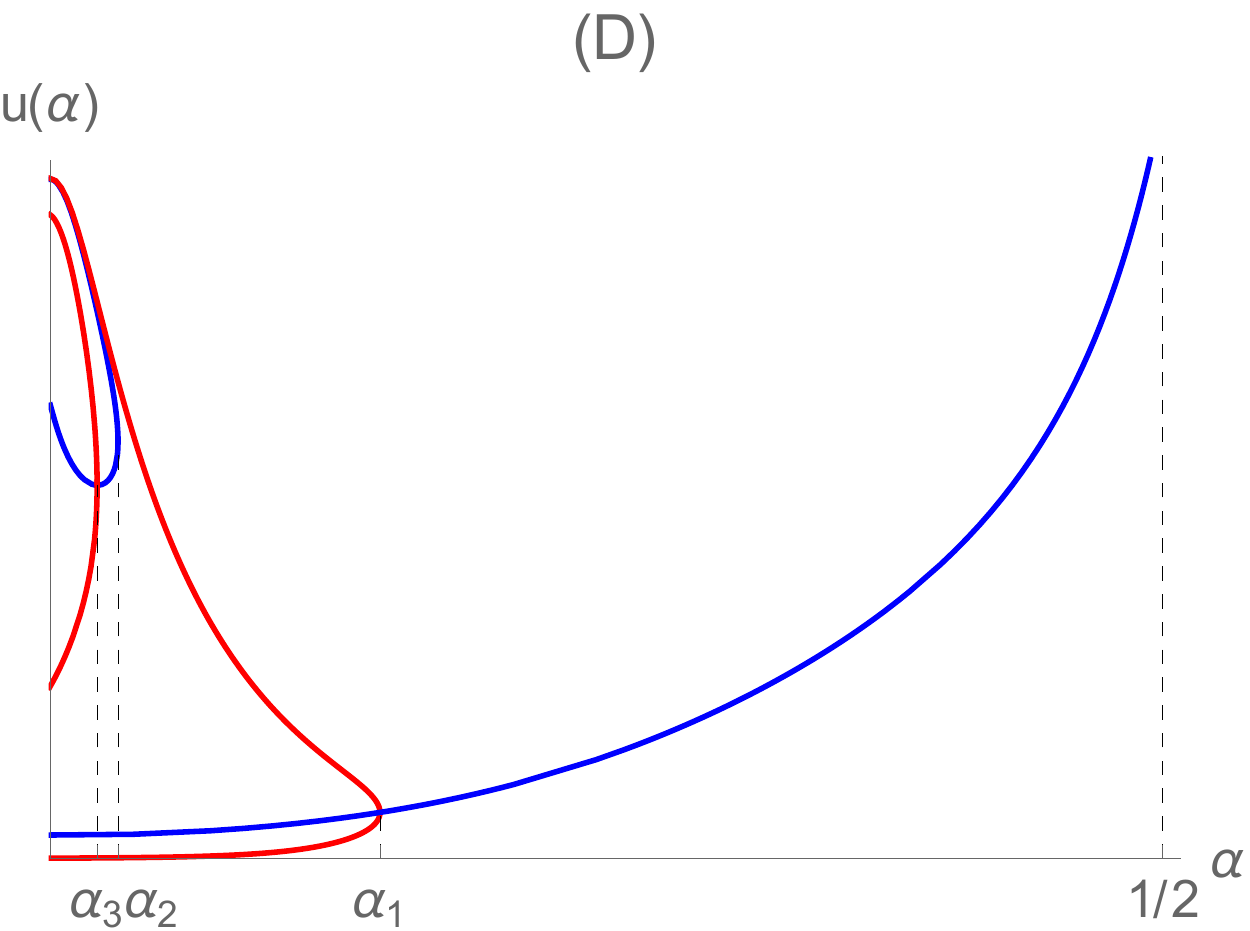}
		\end{tabular} \\
		\caption{Bifurcation diagrams in $\a$ for Problem \eqref{eq11} corresponding to the following cases: (A) $\l\in[\l_1,\l_0)$, (B) $\l\in[\l_2,\l_1)$, (C) $\l\in[\l_3,\l_2)$, (D) $\l\in[\l_4,\l_3)$.} \label{fig5}
	\end{center}
\end{figure}

	\begin{proof}
		(i) Assume that $\l\in[\l_1,\l_0)$. Then, for $\a=0$ Problem \eqref{eq11} admits a unique solution, the constant $\O$, by Theorem \ref{th41}. By Theorem \ref{th45}, there exists a solution for $\a\sim 0$ which converges to it as $\a\da 0$ and is generated by a solution  of $\T_0(\cdot,\a)=1-2\a$, where $\T_0$ is the function defined in \eqref{eq416}. By extending this function by $\tau_1(x,\a)$ for $x\geq x_h(\a)$ (we still denote such an extension by $\T_0$ in order to maintain the notation as simple as possible), since $\T_0$ is continuous and depends continuously on $\a$, Proposition \ref{pr32}(iii) and (v) ensure that such a solution exists for every $\a\in[0,1/2)$ and depends continuously on $\a$, giving the existence of the principal curve in the bifurcation diagram. The fact that it bifurcates from $+\infty$ at $\a=1/2$ follows again by Proposition \ref{pr32}(v), since $1-2\a\da 0$ as $\a\ua1/2$.
		
		(ii) Assume that $\l\in[\l_2,\l_1)$ and consider the set
		\begin{equation*}
		A_{1}:=\{\a\in(0,1/2):\tau_{1}(x_t(\tilde\a),\tilde\a)<1-2\tilde\a \text{ for all $0<\tilde\a<\a$}\}.
		\end{equation*}
		From \eqref{eq411} and \eqref{eq412} we know that $A_{1}\neq\emptyset$ for $\a\sim 0$, and we can therefore define
		\[\a_{1}:=\sup A_{1}>0. \]
		On the other hand, by continuity, we have
		\begin{equation*}
		\liminf_{\a\ua1/2}\tau_{1}(x_t(\a),\a)=\tau_s(x_t(1/2),1/2)>0,
		\end{equation*}
		where $\tau_s$ is given by \eqref{eq35}. Thus, since $1-2\a\da 0$ as $\a\ua1/2$,  $\a\notin A_{1}$ for $\a\sim 1/2$, implying that $\a_{1}<1/2$.
		
		
		It only remains to show the global structure of the bifurcation diagrams and the multiplicity of solutions of Problem \eqref{eq11}. By \eqref{eq411}, \eqref{eq412} and Proposition \ref{pr32}(iii)
		\[\tau_1^{-1}(1-2\a)\cup\tau_2^{-1}(1-2\a) \]
		has at least 3 elements for $\a\sim 0$ forming 3 branches which perturb, by Theorem \ref{th45}, from the corresponding solutions for $\a=0$. By definition of $\a_1$,
		$\tau_{1}(x_t(\a_1),\a_1)=1-2\a_1$ and, since the time maps $\tau_1$ and $\tau_2$ meet at $x=x_t(\a)$,
		such branches meet at $\a=\a_1$ giving rise to a bifurcation point. Moreover, by the definition of $\a_1$ and Proposition \ref{pr32}(v),  one of them can be continued for $\a>\a_1$, forming the principal curve as in part (i).

		(iii) Assume that $\l\in[\l_{n+1},\l_n)$ with $n=2k+1$, $k\in\N^*$. Following the same lines as before, we construct the bifurcation points at the values of $\a$ for which the corresponding time maps meet at $x=x_t(\a)$ taking exactly the value $1-2\a$: consider, for $j\in\N$, the sets
		\begin{equation*}
			A_{2j+1}:=\left\{\a\in(0,1/2):\tau_{2j+1}(x_t(\tilde\a),\tilde\a)<1-2\tilde\a \text{ for all $0<\tilde\a<\a$}\right\},
			\end{equation*}
			observe that \eqref{eq411} and \eqref{eq412} ensure $A_{2j+1}\neq\emptyset$ for $\a\sim 0$ for every $j\in\{0,\ldots,k\}$, and, for such values of $j$, set
			\[\a_{2j+1}:=\sup A_{2j+1}>0. \]
			Proposition \ref{pr32}(ii) gives that $\a_{2j_2+1}<\a_{2j_1+1}$ for $0\leq j_1<j_2\leq k$ and, as a consequence, for every $j\in\{0,\ldots,k\}$, $\a_{2j+1}\leq\a_1<1/2$, as above. 

				In addition, for $j\in\N^*$, we consider the sets
				\begin{equation*}
					B_{2j}:=\left\{\a\in(0,1/2):\inf_{(0,x_h(\tilde\a))}\T_{j}(\cdot,\tilde\a)\leq 1-2\tilde\a \text{ for all $0<\tilde\a<\a$}\right\},
				\end{equation*}
				where $\T_j$ are the functions defined in \eqref{eq416}. As before, for $j\in\{1,\ldots,k\}$ these sets are non-empty for $\a\sim 0$, and we can define
				\[\a_{2j}:=\sup B_{2j}>0, \]
				which, due to Proposition \ref{pr32}(ii), satisfy $\a_{2j_2}<\a_{2j_1}$ for $1\leq j_1<j_2\leq k$.
				In addition, by introducing, for every $j\in\N$,
				\begin{equation*}
					\tilde\T_j(x,\a):=
					\begin{cases}
						\tau_{2j+1}(x,\a) & \text{ if $0<x< x_t$,} \\
						\tau_{2j+2}(x,\a) & \text{ if $x_t\leq x< x_h$,}
					\end{cases}
				\end{equation*}
				we have that, for $j\in\N^*$,
				\begin{equation*}
					\liminf_{\a\ua1/2}\inf_{(0,x_h(\a))}\T_{j}(\cdot,\a)\geq\liminf_{\a\ua1/2}\inf_{(0,x_h(\a))}\tilde\T_{0}(\cdot,\a)=\inf_{(0,x_h(1/2))}\tilde\T_{0}(\cdot,1/2)>0,
				\end{equation*}	
				thus $\a_{2j}<1/2$ for every $j\in\{1,\ldots,k\}$.
				Moreover, by \eqref{eq416} and the definition of the $\a_{2j+1}$ given above, we have, for $0<\tilde\a<\a_{2j+1}$,
				\[1-2\tilde\a>\tau_{2j+1}(x_t(\tilde\a),\tilde\a)\geq\inf_{(0,x_h(\tilde{\a}))}\T_j(\cdot,\tilde\a), \]
				thus $\a_{2j+1}\in B_{2j}$ and $\a_{2j+1}\leq\a_{2j}$ for every $j\in\{1,\ldots,k\}$.

				By construction and Proposition \ref{pr32}(iii), at $\a=\a_{2j}$ the bifurcation diagram presents a subcritical turning point, since the equation $\T_j(\cdot,\a)=1-2\a$ has solutions in a neighborhood of $\a_{2j}$ for $\a\leq\a_{2j}$, and no solution for $\a>\a_{2j}$.
				
				Passing to the global structure of the diagrams and the multiplicity of solutions of Problem \eqref{eq11}, by \eqref{eq411}, \eqref{eq412} and Proposition \ref{pr32}(iii)
				\[\T_j^{-1}(1-2\a)\cup\tilde\T_j^{-1}(1-2\a) \]
				has at least 4 elements for every $j\in\{1,\ldots,k\}$ for $\a\sim 0$. They depend continuously (actually differentiably) on $\a$, forming 4 branches of solutions which perturb from the corresponding solutions that the problem admits for $\a=0$, thanks to Theorem \ref{th45}. By construction, three of such branches join at $\a=\a_{2j+1}$, and two can be continued up to $\a=\a_{2j}$. In this way, for every $j\in\{1,\ldots,k\}$, we have constructed in the bifurcation diagram a component $C_j$ which is bounded (by continuity and since $\a_{2j+1}\leq\a_{2j}<1/2$) and contains at least 4 solutions for $\a<\a_{2j+1}$.
				
				Finally, exactly as shown in part (ii),
				\[\T_0^{-1}(1-2\a)\cup\tilde\T_0^{-1}(1-2\a) \]
				has at least 3 elements for $\a\sim 0$ forming an additional component with a bifurcation point at $\a=\a_1$ on the principal curve.
				
				(iv) The proof follows the same lines of part (iii), with the only difference being that, when $\l\in[\l_{n+1},\l_n)$ with $n=2k$, $k\in\N^*$, we are not able to prove the existence of the first bifurcation point $\a_{2k+1}$ of part (iii) and in general
				\[\T_k^{-1}(1-2\a)\cup\tilde\T_k^{-1}(1-2\a) \]
				may contain only 2 elements	for $\a\sim 0$, which give rise to two branches of solutions perturbing from $\a=0$ and matching in a subcritical turning point at $\a_{2k}$.
		\end{proof}
	\end{theorem}
	
	
	As already pointed out in Remark \ref{re46}, in order to obtain exact multiplicity results for Problem \eqref{eq11}, one should establish the global monotonicities of the time maps $\tau_j$. Moreover, in order to obtain the precise bifurcation diagrams in $\a$, one should also study the dependence of such time maps with respect to $\a$. Both these aspects seem out of reach at present. The bifurcation diagrams of Figure \ref{fig5}, which have been computed numerically by using a path-following continuation method applied to a Fourier--Galerkin spectral discretization of Problem \eqref{eq11} and obviously adjust to the patterns described in Theorem \ref{th51}, show what we expect to be the sharpest results that one should be able to obtain for Problem \eqref{eq11}--\eqref{eq12}. Observe that, the bifurcation parameter being a point of the domain, collocation methods, which are much lighter from the computational point of view, cannot be applied to discretize our problem, as it was for example the case in \cite{LMT4} for the Dirichlet problem studied in \cite{LTZ}, and one has to consider a Fourier--Galerkin method, especially for the computation of the derivative of the discretized problem with respect to $\a$, as indicated by \cite{MM}. This makes the computation of the diagrams much more delicate also from the numerical point of view.
	
	We conclude by describing the corresponding results that one can obtain for Problem \eqref{eq11} in the case of an asymmetric piecewise constant weight $a(t)$.
	
	\begin{remark}
		\label{re52}
		It is possible to adapt the analysis of \cite{LT, T15} to study the case in which the weight function $a(t)$ in \eqref{eq11} is asymmetric of the form
		\begin{equation}
			\label{eqasymmetricweight}
			a(t):=\left\{ \begin{array}{ll} -c_0 & \hbox{if}\;\;
				t\in(0,\alpha), \cr
				b & \hbox{if}\;\; t\in[\alpha,1-\alpha], \cr
				-c_1 & \hbox{if}\;\;
				t\in(1-\alpha,1)
			\end{array}\right.
		\end{equation}
		with $0<c_0\neq c_1>0$. To do so, one essentially has to study the dependence on $c$ of the curves $\G_0$ and $\G_1$ and how the time maps introduced in Section \ref{section3} change in the asymmetric situation, considering separately the cases $c_0>c_1$ and $c_0<c_1$.
		
		The multiplicity result of Theorem \ref{th41}, related to the purely superlinear problem, is still valid,  while the multiplicity for $\a\sim 0$ follows almost along the same lines of Section \ref{section4}, the only difference lying in the structure of the time maps, but the perturbation result remains similar.
		
		The main difference is related to the structure of the global bifurcation diagrams, since now the secondary bifurcations constructed in Theorem \ref{th51} break, giving rise to additional isolated components, as it can be seen by comparing Figure \ref{fig6} with Figure \ref{fig5}(D).
		
		Moreover, following the steps of \cite{T15}, it is possible to show that, when $c_1\to c_0$ the bifurcation diagrams converge, locally uniformly in the complement of the bifurcation points, to the ones of the symmetric case and that, for $c_0\sim c_1$, imperfect bifurcations occur.
		
			\begin{figure}[ht]
				\begin{center}
					\begin{tabular}{cc}
						\includegraphics[scale=0.52]{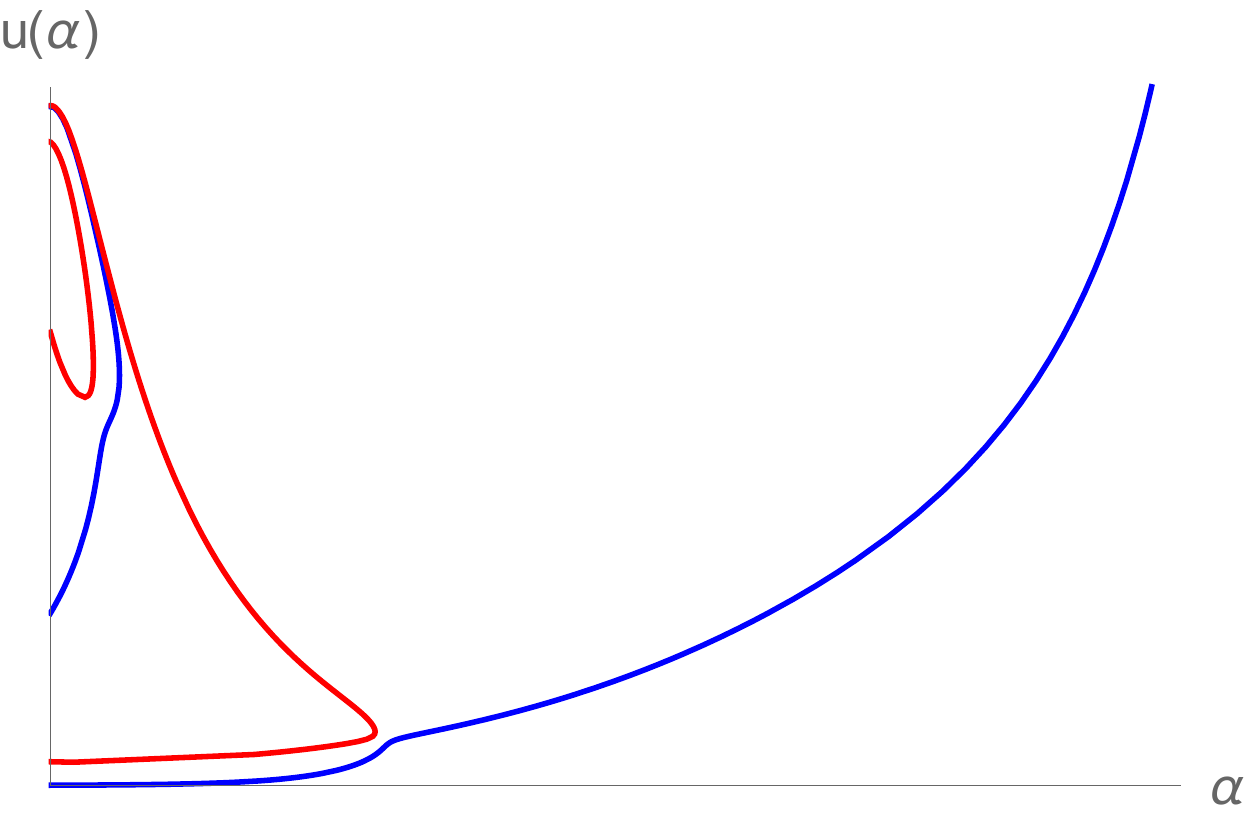} & 			\includegraphics[scale=0.52]{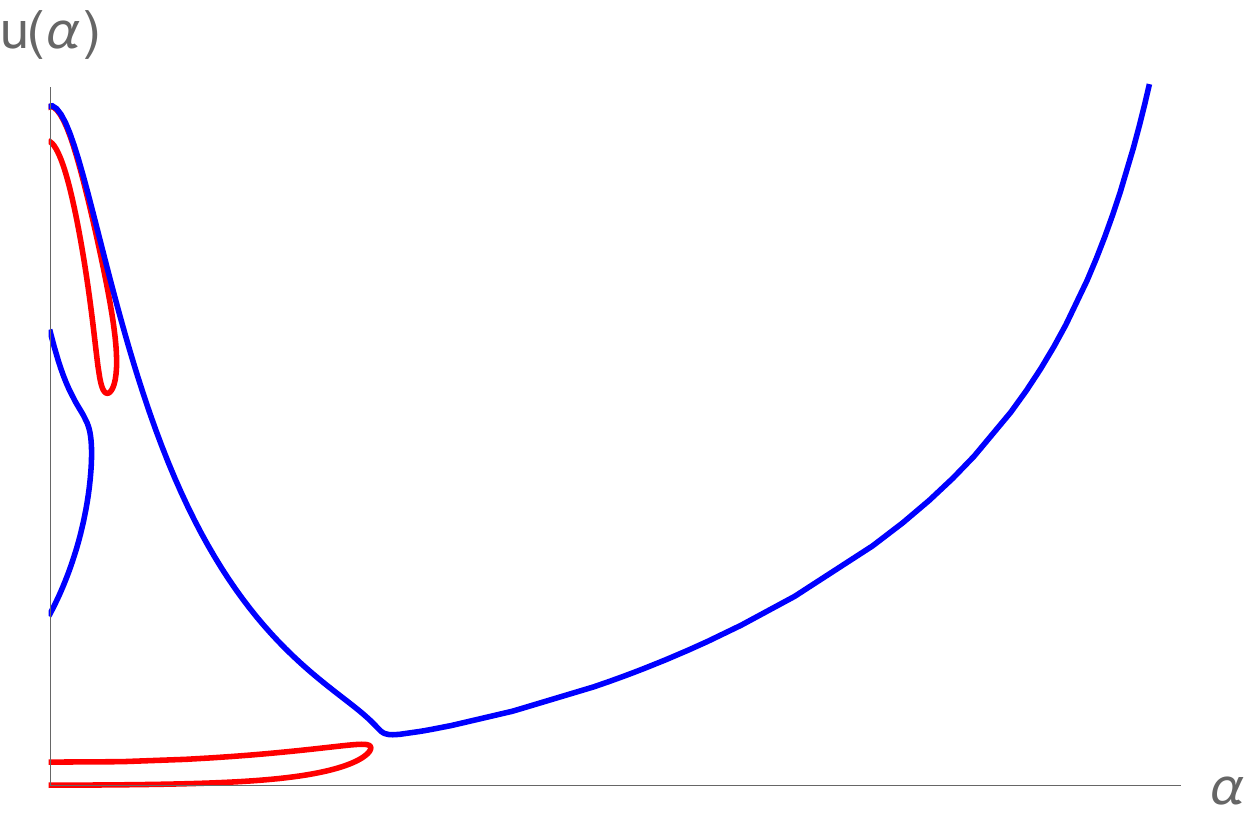}
					\end{tabular} \\
					\caption{Bifurcation diagrams in $\a$ for Problem \eqref{eq11} with weight \eqref{eqasymmetricweight} and $\l\in[\l_4,\l_3)$: case $c_0>c_1$ (left) and case $c_0<c_1$ (right).} \label{fig6}
				\end{center}
			\end{figure}
		\end{remark}
		
%
%
%
%

\section*{Acknowledgments}
I wish to thank Prof. J. L\'opez-G\'omez for the enlightening discussions during the preparation of this work, in particular for the idea of using $\a$ as the main bifurcation parameter, and Prof. M. Molina-Meyer for the suggestions in the implementation of spectral methods for obtaining the bifurcation diagrams through  numerical continuation.

This research has been supported by the European Research
Council under the European Union's Seventh Framework Programme (FP/2007--2013) /
ERC Grant Agreement n.321186 - ReaDi \lq\lq Reaction-Diffusion E\-qua\-tions, Propagation and
Modelling", by the Region \^Ile-de-France under DIM 2014 \lq\lq Diffusion heterogeneities in different spatial dimensions and applications to ecology and medicine" through the Institute of Complex Systems of Paris \^Ile-de-France, by the French National Research Agency under the project \lq\lq NONLOCAL\rq\rq (ANR-14-CE25-0013), and by the Spanish Ministry of Economy, Industry and Competitiveness through project MTM2015-65899-P and contract Juan de la Cierva Incorporaci\'on IJCI-2015-25084.

\end{document}